%% file: thesis.tex
\pdfoutput=1

\documentclass[12pt,a4paper,hidelinks]{report}
\let\openright=\clearpage

\usepackage[utf8]{inputenc}

\usepackage{lmodern}

\usepackage{amsmath}        
\usepackage{amsfonts}       
\usepackage{amsthm}         
\usepackage{amssymb} 
\usepackage{bbding}         
\usepackage{bm}             
\usepackage{graphicx}       
\usepackage{fancyvrb}       
\usepackage[nottoc]{tocbibind} 
\usepackage{dcolumn}        
\usepackage{booktabs}       
\usepackage{paralist}       
\usepackage[usenames]{xcolor}  

\usepackage{tikz-cd}		
\usepackage{enumitem}

\def\ThesisTitle{Combinatorial Properties of Metrically Homogeneous Graphs}

\def\ThesisAuthor{Matěj Konečný}

\def\YearSubmitted{2018}

\def\Department{Department of Applied Mathematics}

\def\DeptType{Department}

\def\Supervisor{Mgr. Jan Hubi\v cka, Ph.D.}

\def\SupervisorsDepartment{Department of Applied Mathematics}

\def\StudyProgramme{Computer Science}
\def\StudyBranch{General Computer Science}

\def\Dedication{%
I would like to thank all the DocCourse people for the fun we had studying the metrically homogeneous graphs, both Honza Hubička and Professor Nešetřil for all the help and support they provided me and Honza for allowing me to use some of his pictures in this thesis. I would also like to thank my girlfriend Inka for her sympathy and all the kind encouragements. And last but not least, I would like to thank the Intel\textsuperscript{®} HD Graphics integrated graphic card for not being capable of running any fun games and thus granting me enough time to write this thesis.
}

\def\Abstract{%
Ramsey theory looks for regularities in large objects. Model theory studies algebraic structures as models of theories. The structural Ramsey theory combines these two fields and is concerned with Ramsey-type questions about certain model-theoretic structures. In 2005, Ne\v set\v ril initiated a systematic study of the so-called \emph{Ramsey classes of finite structures}. This thesis is a contribution to the programme; we find Ramsey expansions of the primitive 3-constrained classes from Cherlin's catalogue of metrically homogeneous graphs. A key ingradient is an explicit combinatorial algorithm to fill-in the missing distances in edge-labelled graphs to obtain structures from Cherlin's classes. This algorithm also implies the \emph{extension property for partial automorphisms} (\emph{EPPA}), another combinatorial property of classes of finite structures.
}

\def\Keywords{%
{Metric space}, {Ramsey theory}, {Homogeneous structure}, {Ramsey expansion}
}

\include{macros}

\begin{document}
\include{title}


\tableofcontents

\include{preface}
\include{history}
\include{catalogue}
\include{preliminaries}
\include{metricspaces}
\include{magiccompletion}
\include{ramseyeppa}

\include{epilog}

\include{bibliography}

\openright
\end{document}

%% file: macros.tex
\makeatletter
\def\@makechapterhead#1{
  {\parindent \z@ \raggedright \normalfont
   \Huge\bfseries \thechapter. #1
   \par\nobreak
   \vskip 20\p@
}}
\def\@makeschapterhead#1{
  {\parindent \z@ \raggedright \normalfont
   \Huge\bfseries #1
   \par\nobreak
   \vskip 20\p@
}}
\makeatother

\theoremstyle{plain}
\newtheorem{theorem}{Theorem}
\numberwithin{theorem}{chapter}

\newtheorem{lemma}[theorem]{Lemma}

\newtheorem{prop}[theorem]{Proposition}
\newtheorem{corollary}[theorem]{Corollary}
\newtheorem{observation}[theorem]{Observation}

\theoremstyle{definition}
\newtheorem{definition}[theorem]{Definition}
\newtheorem{example}[theorem]{Example}
\newtheorem{question}{Question}

\theoremstyle{remark}

\newtheorem*{remark*}{Remark}

\DefineVerbatimEnvironment{code}{Verbatim}{fontsize=\small, frame=single}

\newcommand{\Forb}{\mathop{\mathrm{Forb}}\nolimits}
\newcommand{\str}[1]{\mathbf{#1}}
\def\Fraisse{Fra\"{\i}ss\' e}

\def\rel#1#2{R_{\mathbf{#1}}^{#2}}

\newcommand{\arity}[1]{a(\rel{}{#1})}

\def\Age{\mathop{\mathrm{Age}}\nolimits}

\newcommand{\overbar}[1]{\mkern 1.5mu\overline{\mkern-1.5mu#1\mkern-1.5mu}\mkern 1.5mu}

\def\Aclass{\mathcal A^\delta_{K_1,K_2,C_0,C_1}}

%% file: title.tex


\pagestyle{empty}
\begin{center}


\vspace{-8mm}
\vfill

{\bf\Large BACHELOR THESIS}

\vfill

{\LARGE\ThesisAuthor}

\vspace{15mm}

{\LARGE\bfseries\ThesisTitle}

\vfill

\Department

\vfill

\begin{tabular}{rl}

Supervisor of the bachelor thesis: & \Supervisor \\
\noalign{\vspace{2mm}}
Study programme: & \StudyProgramme \\
\noalign{\vspace{2mm}}
Study branch: & \StudyBranch \\
\end{tabular}

\vfill

Prague \YearSubmitted

\end{center}

\newpage



\openright
\pagestyle{plain}
\pagenumbering{roman}
\vglue 0pt plus 1fill

\noindent
I declare that I carried out this bachelor thesis independently, and only with the cited
sources, literature and other professional sources.

\medskip\noindent
I understand that my work relates to the rights and obligations under the Act No.~121/2000 Sb.,
the Copyright Act, as amended, in particular the fact that the Charles
University has the right to conclude a license agreement on the use of this
work as a school work pursuant to Section 60 subsection 1 of the Copyright Act.

\vspace{10mm}

\hbox{\hbox to 0.5\hsize{%
In ........ date ............	
\hss}\hbox to 0.5\hsize{%
signature of the author
\hss}}

\vspace{20mm}
\newpage


\openright

\noindent
\Dedication

\newpage


\openright

\vbox to 0.5\vsize{
\setlength\parindent{0mm}
\setlength\parskip{5mm}

Title:
\ThesisTitle

Author:
\ThesisAuthor

\DeptType:
\Department

Supervisor:
\Supervisor, \SupervisorsDepartment

Abstract:
\Abstract


\vss}

\newpage

\openright
\pagestyle{plain}
\pagenumbering{arabic}
\setcounter{page}{1}

%% file: preface.tex
\chapter*{Introduction}
\addcontentsline{toc}{chapter}{Introduction}

A metrically homogeneous graph is a countable graph which gives rise to a homogeneous metric space when one computes the distances between all vertices. Cherlin~\cite{Cherlin2013} recently gave a catalogue of such graphs in terms of classes of finite subspaces of the metric spaces. In this thesis, we shall study the following question: ``Given a graph $G$ with edges labelled by $1,2,\ldots,\delta$, when is it possible to add the remaining edges and their labels such that the resulting structure, understood as a metric space, belongs to one of Cherlin's classes?'' Although the question might sound somewhat arbitrary, it is motivated by (and important for) applications in Ramsey theory and combinatorial model theory in general.

This thesis is based on joint work~\cite{Aranda2017a, Aranda2017c, Aranda2017} with Andr\'es Aranda, David Bradley-Williams, Eng Keat Hng, Jan Hubi\v{c}ka, Miltiadis Karamanlis, Michael Kompatscher and Micheal Pawliuk, which was done during and after the Ramsey DocCourse 2016 programme. Some of the results were also obtained independently by Rebecca Coulson in her PhD thesis~\cite{Coulson}. Ramsey expansion of the case $(\delta, K_1, K_2, C_0, C_1) = (3, 1, 3, 8, 9)$ was obtained by Soki\'c~\cite{Sokic2017}.

\section*{Organization of the thesis}
In Chapter~\ref{ch:history}, we briefly review the history of Ramsey theory and of the study of homogeneous structures. In Chapter~\ref{ch:catalogue} we present the relevant parts of Cherlin's catalogue of metrically homogeneous graphs, namely the primitive 3-constrained cases. Then, in Chapter~\ref{ch:preliminaries}, we define all necessary notions and present the results of~\cite{Hubicka2016} on multiamalgamation classes, which we need in Chapter~\ref{ch:ramseyeppa}.

Chapter~\ref{ch:metricspaces}, which serves as a warm-up, is concerned with the Ramsey property of finite metric spaces with distances $\{0,1,\ldots, n\}$ for some positive integer $n$. It illustrates how one uses Theorem~\ref{thm:hn} to prove the Ramsey property. In Chapter~\ref{ch:magiccompletion} we define a procedure to fill in the missing distances in partial metric spaces and prove that it produces a metric space from Cherlin's catalogue whenever it is possible. This result is then used in Chapter~\ref{ch:ramseyeppa} to find Ramsey expansions and prove the so-called EPPA of metrically homogeneous graphs.

\section*{Outline of the results}
We only present the necessary notions very briefly, they will be defined properly in the following chapters.

A structure is \emph{homogeneous} if every isomorphism between its finite substructures extends to an automorphism of the whole structure (see Section~\ref{sec:homog}).

A metric space with distances from $\{0,1,\ldots, \delta\}$, $\delta\geq 1$, is called a \emph{$\delta$-valued metric space}. One can view a $\delta$-valued metric space as a complete graph with edges labelled by $\{1,2,\ldots,\delta\}$ (a \emph{$\delta$-edge-labelled complete graph}) such that it does not contain any \emph{non-metric triangle}, that is, a triangle with edges labelled $a,b,c$ such that $a>b+c$.

A metrically homogeneous graph is a countable connected graph such that if one computes the distances between each two vertices, the resulting structure is then homogeneous as a metric space. Cherlin~\cite{Cherlin2013} (over 600 pages) gave a catalogue of metrically homogeneous graphs. A significant and key part of it are the primitive 3-constrained classes, which are homogeneous $\delta$-valued metric spaces such that some other triangles are also forbidden. Informally, the other forbidden triangles come from certain families of triangles of small odd perimeter, triangles of odd perimeter where two edges are much longer than the third one and triangles of large perimeter.

Let $\mathcal C$ be a class of $\delta$-valued metric spaces (for example one of Cherlin's classes) and let $\str G = (V, E, d)$ be an \emph{$\delta$-edge-labelled graph}, i.e. $(V,E)$ is a graph and $d$ is a function from $E$ to $\{1,2,\ldots,\delta\}$. Then we say that $\str G' = (V, {V\choose 2}, d')$ is a \emph{completion} of $\str G$ in $\mathcal C$ if $d'$ is a function from ${V\choose 2}$ to $\{1,\ldots,\delta\}$, $d'|_E = d$ and $(V,d')\in \mathcal C$. This says that one can add the missing distances to $\str G$ and get a metric space from $\mathcal C$.

The main theorem of the influential paper by Hubi\v cka and Ne\v set\v ril~\cite{Hubicka2016}, here stated in a weaker form as Theorem~\ref{thm:hn}, says very roughly the following: If one can characterize the $\delta$-edge-labelled graphs which admit a completion into $\mathcal C$ by a finite set of forbidden homomorphic images, then $\mathcal C$ is Ramsey. (What it means to have the Ramsey property, or be Ramsey, is defined in Chapter~\ref{ch:history}.)

For standard $\delta$-valued metric spaces one can use the \emph{shortest path completion}, which sets the distance between each two vertices to be the minimum over the lengths of all paths connecting the two vertices or $\delta$ if there are no such paths or all paths are longer that $\delta$. One can prove that the shortest path completion of a graph is a metric space if and only if the graph contains no (homomorphic image of a) non-metric cycle, i.e. a cycle with edges of lengths $\ell, a_1,\ldots,a_k$ with $\ell > \sum_{i=1}^k a_i$. As there are only finitely many such cycles --- their number is a function of $\delta$ --- Theorem~\ref{thm:hn} can then be used to prove that the class of all linearly ordered $\delta$-valued metric spaces has the Ramsey property.\footnote{Why one needs a linear order will become clear from the precise statement of Theorem~\ref{thm:hn}.}

For Cherlin's spaces, the nature of the argument is similar, but technically much more challenging. One could try to use the shortest path completion for all Cherlin's spaces. This, however, does not work as the shortest path completion sets each distance to be as large as possible and in Cherlin's spaces triangles with long perimeter are forbidden. To overcome this, instead of seting each distance to be as large as possible, our completion procedure (called the \emph{magic completion}) sets each distance as close as possible to some suitably chosen \textit{magic parameter} $M$ while satisfying some local constraints.

We first need to prove that the magic completion is correct (that is, produces a metric space in a given Cherlin's class whenever it is possible). The argument is much more difficult than for the shortest path completion. To be able to carry it out, we introduce a refined version of the magic completion called the \emph{magic completion algorithm} which works as follows:

First, given a Cherlin's class and a magic parameter $M$, we define an operation $\oplus$ on the set $\{1,\ldots,\delta\}$ describing that the magic completion algorithm will use the distance $a\oplus b$ to complete the \emph{fork} $a,b$ (i.e. an induced path of length $2$ with distances $a$ and $b$). The operation $\oplus$ will satisfy (and up to some special cases will be defined by) the property that $a\oplus b$ is the distance closest to $M$ such that the triangle with distances $a,b,a\oplus b$ is allowed.

Then we define a permutation $\pi_M$ of the set $\{1, \ldots, \delta\}$ such that $\pi_\delta = M$. The magic completion algorithm has $\delta$ stages, in stage $i$ it looks at all forks and if the $\oplus$-sum of some fork is $\pi_i$, it sets the missing distance in the fork to $\pi_i$. Notice that for $a\oplus b = \min(\delta, a+b)$ and $\pi_i = i$ the magic completion algorithm gives precisely the shortest path completion.

Having this step-by-step completion algorithm, we will then check that whenever it produces a forbidden triangle, then the input graph cannot be completed into the given class. It is a lot of technical work and inequality checking --- because of the way Cherlin's spaces are defined --- and uses some of the nice properties of the magic completion algorithm such as being as close to $M$ as possible for each pair of vertices independently.

Then, using the magic completion algorithm, we prove that the family of forbiden homomorphic images for each Cherlin's class is indeed finite: Take an arbitrary incomplete $\delta$-edge-labelled graph $\str G$ and take its magic completion. Either we get something from the class, or the resulting structure contains a forbidden triangle. Then we backtrack the magic completion algorithm and look for a substructure of $\str G$ which caused the magic completion to produce the forbidden triangle. In each stage, at worst every edge of the structure was added due to a fork, so the number of edges at most doubles with each backtrack-stage. Therefore, every forbidden substructure has at most a bounded number of edges and vertices, hence there are only finitely many of them.

Combining the last paragraph and Theorem~\ref{thm:hn} it follows that for all suitable parameters the class of all linearly ordered Cherlin's metric spaces has the Ramsey property.

We also prove that all Cherlin's classes have the \emph{extension property for partial automorphisms} (called \emph{EPPA}, see Definition~\ref{defn:eppa}). It follows from Theorem~\ref{thm:herwiglascar}, the fact that there are only finitely many forbidden homomorphic images and furthermore the fact that the magic completion algorithm is \emph{automorphism preserving} which means that whenever $g$ is an automorphism of an incomplete $\str G$ then $g$ is also an automorphism of the magic completion of $\str G$. Our paper~\cite{Aranda2017} was one of the first instances when the Ramsey property and EPPA were proved simultaneously using essentially the same techniques.

%% file: history.tex
\chapter{History and background}\label{ch:history}
In this thesis, a \emph{relational language} $L$ is a collection $\{\rel{}{i}\}_{i\in I}$ of relational symbols with given arities. As we only work with relational languages, we simply write \emph{language}. An \emph{$L$-structure} $\str M$ is then a tuple $(M, \{\rel{M}{i}\}_{i\in I})$, where $M$ is a set and $\rel{M}{i}$ is an interpretation of $\rel{}{i}$ in $\str M$. Notationally, we shall distinguish structures from their underlying sets by typesetting structures in bold font. When the language $L$ is clear from the context, we will use it implicitly. We use $\arity{}$ to denote the arity of $\rel{}{}$.

Let $\str A$ and $\str B$ be $L$-structures. An \emph{embedding} $f\colon\str A\rightarrow \str B$ is an injection $f\colon A\rightarrow B$ such that for all $\rel{}{} \in L$ and for all $(a_1, \ldots, a_{\arity{}})$ it holds that $(a_1, \ldots, a_{\arity{}}) \in \rel{A}{}$ if and only if $(f(a_1), \ldots, f(a_{\arity{}}))\in \rel{B}{}$, i.e. $f$ preserves all relations and non-relations. An \emph{isomorphism} is a bijective embedding, an \emph{automorphism} is an isomorphism $\str A \rightarrow \str A$.

For $L$-structures $\str A$ and $\str B$ with $A\subseteq B$, we say that $\str A$ is a \emph{substructure} of $\str B$ and denote is as $\str A\subseteq \str B$ if the inclusion is an embedding.

If $f\colon X\rightarrow Z$ is a function and $Y\subseteq X$, we denote by $f|_Y$ the restriction of $f$ on $Y$.

For more on model theory see the textbook by Hodges~\cite{Hodges1993}.

\section{Homogeneous structures}\label{sec:homog}
The main reference for this section is the survey on homogeneous structures by~Macpherson~\cite{Macpherson2011}.

We say that an $L$-structure $\str M$ is \emph{homogeneous} (sometimes called \emph{ultrahomogeneous}) if for every finite $\str A,\str B\subseteq \str M$ and every isomorphism $g\colon  \str A\rightarrow \str B$ there is an automorphism $f$ of $\str M$ with $f|_A = g$.

Let $\mathcal C$ be a class of (not necessarily all) $L$-structures and let $\str M$ be an $L$-structure. We say that $\str M$ is \emph{universal} for $\mathcal C$ if for every $\str C\in \mathcal C$ there exists an embedding $\str C\rightarrow \str M$.

\begin{example}
The structure $(\mathbb Q, <)$, where $<$ is the standard linear order of the rationals, is homogeneous and universal for countable linear orders.

\end{example}

In the early 1950s \Fraisse{}~\cite{Fraisse1953, Fraisse1986} noticed that while $(\mathbb Q, <)$ is homogeneous and universal for countable linear orders, $(\mathbb N, <)$ is neither of those (though it is universal for finite linear orders) and as a (very successful) attempt to extract the necessary properties and generalise this phenomenon he proved Theorem~\ref{thm:fraisse}. The extracted properties are summarised in the following definition.

\begin{definition}[Amalgamation property~\cite{Fraisse1953}]\label{defn:amalgamation}
Fix a language $L$ and let $\mathcal C$ be a non-empty class of finite $L$-structures. We say that $\mathcal C$ is an \emph{amalgamation class} if it has the following properties:
\begin{enumerate}
\item $\mathcal C$ is closed under isomorphisms and substructures;
\item $\mathcal C$ has the \emph{joint embedding property (JEP)}: For all $\str B_1, \str B_2\in \mathcal C$ there is $\str C\in \mathcal C$ and embeddings $\beta_1\colon  \str B_1\rightarrow \str C$ and $\beta_2\colon \str B_2\rightarrow \str C$; and
\item $\mathcal C$ has the \emph{amalgamation property (AP)}: For all $\str A, \str B_1, \str B_2\in \mathcal C$ and embeddings $\alpha_1\colon \str A\rightarrow \str B_1$ and $\alpha_2\colon \str A\rightarrow \str B_2$, there is $\str C\in \mathcal C$ and embeddings $\beta_1\colon  \str B_1\rightarrow \str C$ and $\beta_2\colon \str B_2\rightarrow \str C$ such that $\beta_1\circ\alpha_1 = \beta_2\circ\alpha_2$.
\end{enumerate}
It is common in the area to identify members of $\mathcal C$ with their isomorphism types.

An amalgamation class is a \emph{\Fraisse{} class} if it contains only countably many pairwise non-isomorphic structures (which happens for example when the language is countable).
\end{definition}
The joint embedding property intuitively says that for every two structures in $\mathcal C$ there is a structure in $\mathcal C$ containing both of them. If present in $\mathcal C$, their disjoint union could play such a role. But at the other extreme, if $\str B_1 = \str B_2$ then one can also put $\str C \cong \str B_1$.

The amalgamation property roughly says that if one glues two structures $\str B_1$ and $\str B_2$ over a common substructure $\str A$, there is a structure $\str C\in\mathcal C$ which contains this patchwork, see Figure~\ref{fig:amalgamation}. By definition it is possible that the embeddings of $\str B_1$ and $\str B_2$ in $\str C$ overlap by more vertices than just the vertices of $\str A$ and also that $\str C$ has more vertices than just $\beta_1(B_1)\cup \beta_2(B_2)$.

If $\mathcal C$ contains the empty structure, then AP for $\str A$ being the empty structure is precisely JEP.

\begin{definition}\label{defn:strongamalgamation}
We say that a class has the \emph{strong amalgamation property} (or is a \emph{free amalgamation class}) if it has the amalgamation property as in Definition~\ref{defn:amalgamation} such that $C=\beta_1(B_1)\cup\beta_2(B_2)$ and $\beta_1(x_1) = \beta_2(x_2)$ if and only if $x_1\in\alpha_1(\str A)$ and $x_2\in\alpha_2(\str A)$.
\end{definition}
Intuitively, strong amalgamation means that in the amalgam one only glues over the necessary substructure $\str A$ and nothing more.

Let $\str M$ be an $L$-structure. Then the \emph{age of $\str M$} is defined as $$\Age(\str M) = \left\{\str A\mid \str A\text{ is a \textbf{finite} $L$-structure with an embedding $\alpha\colon \str A\rightarrow \str M$}\right\}.$$
Again, it is common to identify the age with the class of all isomorphism types of structures in the age.

\begin{figure}
\centering
\includegraphics{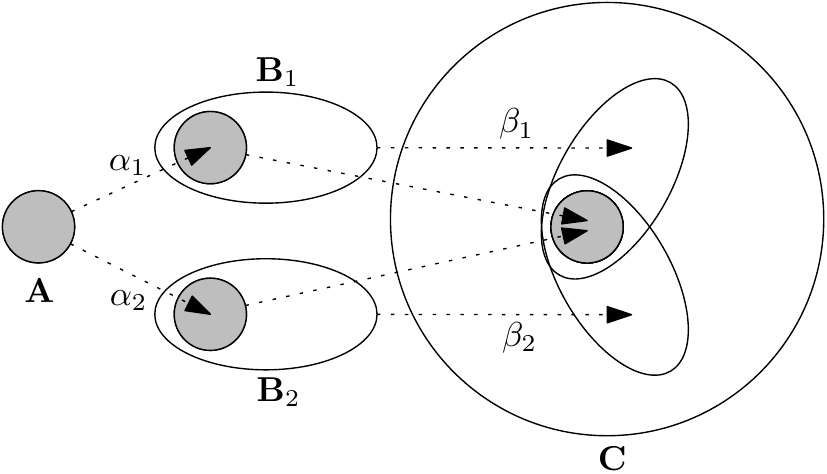}
\caption{An amalgamation of $\str{B}_1$ and $\str{B}_2$ over $\str{A}$.}
\label{fig:amalgamation}
\end{figure}

\begin{theorem}[\Fraisse{}~\cite{Fraisse1953}]\label{thm:fraisse}\leavevmode
\begin{enumerate}
\item Let $\str M$ be a countable homogeneous $L$-structure. Then $\Age\left(\str M\right)$ is a \Fraisse{} class.
\item Let $\mathcal C$ be a \Fraisse{} class. Then there is a countable homogeneous $L$-structure $\str M$ with $\Age(\str M) = \mathcal C$. Furthermore, if $\str N$ is a countable homogeneous $L$-structure with $\Age(\str N) = \mathcal C$, then $\str M$ and $\str N$ are isomorphic.
\end{enumerate}
We call the structure $\str M$ from the second point the \emph{\Fraisse{} limit of $\mathcal C$}. 
\end{theorem}

\Fraisse{}'s theorem gives a correspondence between amalgamation classes and homogeneous structures. Besides the age of $(\mathbb Q, <)$, which is the class of all finite linear orders, there are many more known homogeneous structures and amalgamation classes. A prominent example is for example the random graph (often called the Rado graph), which we understand as a structure with one binary relation $E$ whose age is the class of all finite graphs and which is actually universal for all countable graphs. Or the random triangle-free graph, which is the \Fraisse{} limit of all finite triangle-free graphs and again is universal for all countable triangle-free graphs. An example from a very different area is Hall's universal group~\cite{Hall1959} (a nice exposition is in~\cite{Siniora2}) which, of course, doesn't fall into our framework as we only work with relational structures, but the same theory including an appropriate variant of the \Fraisse{} theorem holds for structures in languages containing both relations and functions. Hall's universal group is universal for all countable locally-finite groups and it is a countable homogeneous group, which means that every isomorphism between finite subgroups can be extended to an automorphism of the whole group.

Another example, which is especially interesting to us, is the Urysohn space, which is a homogeneous separable metric space universal for all countable separable metric spaces. It was constructed by Urysohn in~1924~\cite{Urysohn1927}. For a historical and metric-space-theoretical context, see~\cite{HusekUrysohn2008} where the author compares Urysohn's, Hausdorff's and later Kat\v etov's approaches to the problem.

The Urysohn space $\mathbb U$ is constructed as the completion (in the metric space sense) of the rational Urysohn space $\mathbb U_{\mathbb Q}$, which is a homogeneous countable metric space with rational distances which is universal for all finite metric spaces with rational distances. $\mathbb U_{\mathbb Q}$ is constructed by a procedure in principle not very different from what \Fraisse{} used to prove Theorem~\ref{thm:fraisse}. Hence Urysohn was ahead of \Fraisse{} by roughly 30 years and the whole theory should perhaps be called Urysohn-\Fraisse{} theory instead of just \Fraisse{} theory.

We shall study $\mathbb U_{\mathbb Q}$ and its relatives in more detail in Chapter~\ref{ch:metricspaces} which serves as a warm-up for later chapters.

So far we have seen a couple of examples of homogeneous structures, which are homogeneous for ``good reasons''. On the other hand, $(\mathbb N, <, (U_i)_{i\in \mathbb N})$, the structure consisting of natural numbers with the standard order plus infinitely many unary relations such that $U_i^{\mathbb N} = \{i\}$ (each vertex gets its own unary relation), is also homogeneous, but for ``stupid reasons'', namely the complete lack of isomorphisms between substructures.

\subsection{Classification results}
Homogeneous structures are studied from several different perspectives. In this thesis we promote the combinatorial one. Other possible perspective is one of group theory: The automorphism groups of homogeneous structures are very rich (unless --- of course --- the structure is for example $(\mathbb N, <, (U_i)_{i\in \mathbb N})$). There are many results and notions connected to automorphism groups of homogeneous structures and a survey by Cameron~\cite{Cameron1999} serves as a very good starting reference. In model theory, homogeneous structures are studied for example from the stability point of view, see~\cite{Macpherson2011} for details. And last but not least, the automorphism group can be equipped with a natural topology --- the \emph{pointwise-convergence topology} --- and then studied from the point of view of topological dynamics. 

But the initial direction after the \Fraisse{} theorem was on classification. In this section, we briefly and partially overview some classification results as their statements are usually very long (and their proofs all the more so). We start with a theorem of Lachlan and Woodrow on the classification of all countably infinite homogeneous (undirected) graphs.

\begin{theorem}[Classification of countably infinite homogeneous graphs~\cite{Lachlan1980}]\label{thm:lachlanwoodrow}
Let $\str G$ be a countably infinite homogeneous undirected graph. Then $\str G$ or $\overbar{\str G}$ (the complement of $\str G$) is one of the following:
\begin{enumerate}
\item The random graph $\str R$ (i.e. the \Fraisse{} limit of the class of all finite graphs);
\item the \emph{generic} (that is, universal and homogeneous) $K_n$-free graph for some finite clique $K_n$, which is the \Fraisse{} limit of the class of all finite $K_n$-free graphs;	or
\item the disjoint union of complete graphs of the same size (either an infinite union of $K_n$'s for some $n<\infty$, or a finite or infinite union of $K_\omega$'s).
\end{enumerate}
\end{theorem}

This theorem implies that there are only countably many countable homogeneous graphs, which is in contrast with an earlier result of Henson~\cite{Henson1972} who found $2^{\aleph_0}$ non-isomorphic countable homogeneous directed graphs. They are analogues of the $K_n$-free graphs, but while forbidding $K_n$ and $K_m$ is the same as forbidding $K_{\min(m, n)}$, Henson is forbidding tournaments and one can construct infinite sets of pairwise \textit{incomparable} tournaments. Cherlin~(\cite{Cherlin1998}), more than 20 year later, then gave a full classification of countably infinite homogeneous directed graphs. The proof takes more than 170 pages and even the list itself is too long and complicated for our small historical overview; it contains for example the much older classification of homogeneous partial orders~(\cite{Schmerl1979}).

There are many more classification results. In Chapter~\ref{ch:catalogue} we present part of Cherlin's catalogue of metrically homogeneous graphs which by its complexity (and the proof length) substantially surpasses his work on homogeneous digraphs.



\section{Ramsey theory}
Surveying the rich history of Ramsey theory could easily be a topic for more than one thesis, but not a very good topic as there are many good references already. Here we only mention the results which are directly relevant to our problem, for more see for example Ne\v set\v ril's Chapter in~\cite[Ch. 25]{Graham1995}. Some of the more recent developments in the structural Ramsey theory have been surveyed for example by Bodirsky~\cite{Bodirsky2015}, Nguyen Van Th\'e~\cite{NVT14} and Solecki~\cite{Solecki2013}.

In~1930 Ramsey published a paper where he proves the folowing theorem (which we state in today's language, by $[n]$ we mean the set $\{0, 1, \ldots, n-1\}$, and, for a set $A$, by ${A\choose p}$ we mean the set of all $p$-elements subsets of $A$):
\begin{theorem}[Ramsey's theorem~\cite{Ramsey1930}]
For every triple of natural numbers $n,p,k$ there is $N$ such that the following holds:

For every colouring $c\colon  {[N]\choose p} \rightarrow [k]$ there is an $n$-element subset $H\in {[N]\choose n}$ such that $c|_{H\choose p}$ is constant.

\end{theorem}

There are many Ramsey-like theorems. This thesis is a contribution to the structural Ramsey theory, where instead of finite subsets of $\mathbb N$ one is working with structures and instead of $p$-tuples one is colouring embeddings.

\subsection{Structural Ramsey theory}

\begin{definition}[Ramsey property]
Let $L$ be a finite relational language and $\mathcal C$ be a class of finite $L$-structures. 

For $\str A, \str B\in \mathcal C$ we denote by ${\str B\choose \str A}$ the set of all embeddings $f\colon \str A\rightarrow \str B$.

We say that $\mathcal C$ has the \emph{Ramsey property} (or \emph{is Ramsey}) if for every $\str A, \str B\in \mathcal C$ and every $k\geq 1$ there is $\str C\in \mathcal C$ such that for every colouring $c\colon {\str C\choose \str A}\rightarrow [k]$ there is an embedding $\beta\in {\str C\choose \str B}$ such that for every $\alpha_1, \alpha_2\in {\str B\choose \str A}$ it holds that $c(\beta\circ \alpha_1) = c(\beta\circ \alpha_2)$ (``all copies of $\str A$ in $\str B$ are monochromatic'').
\end{definition}

In~1977 Ne\v set\v ril and R\"odl and independently in~1978 Abramson and Harrington proved the following:
\begin{theorem}[Ne\v set\v ril-R\"odl~\cite{Nevsetvril1977,Nevsetvril1977b}, Abramson-Harrington~\cite{Abramson1978}]\label{thm:nrgraph}
The class of all linearly ordered finite graphs is Ramsey.
\end{theorem}
Note that the language is $L=\{E, \leq\}$ and the embeddings also have to preserve the order.

The techniques of Ne\v set\v ril and R\"odl actually prove much more. Let $L = \{\rel{}{1}, \rel{}{2}, \ldots, \rel{}{k}\}$ be a finite relational language. We say that an $L$-structure $\str M$ is \emph{irreducible} if for every two vertices $x\neq y \in M$ there is a relation $\rel{}{i}$ and a tuple $(z_1, z_2, \ldots, z_{\arity{i}}) \in \rel{}{i}$ such that $x = z_a$ and $y = z_b$ for some $a,b$. If $L=\{E\}$ is the graph language, then $\str G$ is irreducible if and only if $\str G$ is the complete graph.

Now we state what is known as the Ne\v set\v ril-R\"odl theorem:
\begin{theorem}[Ne\v set\v ril-R\"odl~\cite{Nevsetvril1977,Nevsetvril1977b}]\label{thm:nr}
Let $L$ be a finite relational language and let $\mathcal F$ be a finite collection of irreducible finite $L$-structures. Define $\overrightarrow{\mathcal F}$ to be the collection of all linear orderings of structures from $\mathcal F$. Then the class of all linearly ordered finite $L$-structures $\str M$ such that there is no $\str F\in \overrightarrow{\mathcal F}$ with an embedding $\str F\rightarrow \str M$ is Ramsey.
\end{theorem}

Theorem~\ref{thm:nrgraph} is a direct consequence of Theorem~\ref{thm:nr}. Theorem~\ref{thm:nr} also implies that the class of ordered $K_n$-free graphs is a Ramsey class. Thanks to the Ne\v set\v ril-R\"odl theorem, we now know a lot of Ramsey classes. We also know that every Ramsey class must consist of rigid structures (otherwise, if a structure $\str A$ had a nontrivial automorphism, we could colour its embedding according to the automorphisms). The following observation of Ne\v set\v ril from~1989 gives (under a mild assumption) a strong necessary condition for Ramsey classes and also connects the Ramsey theory and the theory of homogeneous structures:

\begin{theorem}[Ne\v set\v ril~\cite{Nevsetvril1989a,Nevsetril2005}] \label{thm:ramseyamalg}
Let $\mathcal C$ be a Ramsey class of finite structures with the joint embedding property. Then $\mathcal C$ has the amalgamation property.
\end{theorem}
\begin{figure}
\centering
\includegraphics{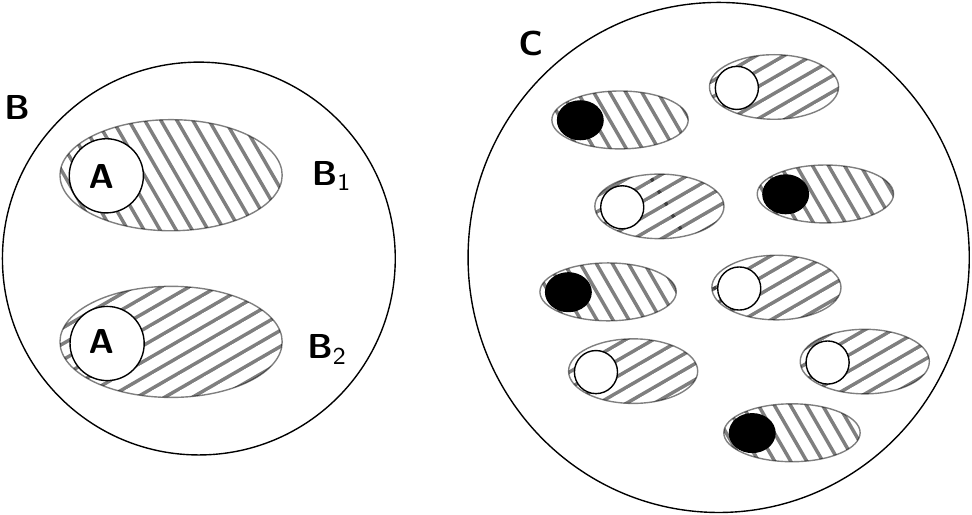}
\caption{Illustration of the proof of Theorem~\ref{thm:ramseyamalg}. Copies of $\str A$ from $\str B_1$ are coloured black, copies of $\str A$ from $\str B_2$ are coloured white.}
\label{fig:ramseyamalg}
\end{figure}
\begin{proof}
We need to show that for every $\str A,\str B_1,\str B_2\in \mathcal C$ and embeddings $\alpha_1\colon  \str A\rightarrow \str B_1$ and $\alpha_2\colon  \str A\rightarrow \str B_2$ there is $\str C\in \mathcal C$ and embeddings $\beta_1\colon  \str B_1\rightarrow \str C$ and $\beta_2\colon  \str B_2\rightarrow \str C$ such that $\beta_1\circ \alpha_1 = \beta_2\circ \alpha_2$.

Let $\str B$ be a joint embedding of $\str B_1$ and $\str B_2$ and take such $\str C\in \mathcal C$ that $\str C \longrightarrow (\str B)^\str A_2$. We will prove that $\str C$ is the amalgam we are looking for.

Assume the contrary which means that there is no embedding $\alpha\colon  \str A\rightarrow \str C$ with the property that there are embeddings $\beta_1\colon  \str B_1\rightarrow \str C$ and $\beta_2\colon  \str B_2\rightarrow \str C$ such that $\beta_i\circ\alpha_i = \alpha$ for $i\in\{1,2\}$. Hence, for every $\alpha\colon  \str A\rightarrow \str C$ there is at most one such embedding $\beta_i\colon  \str B_i\rightarrow \str C$. Define the colouring $c\colon  {\str C\choose \str A} \rightarrow \{0, 1\}$ by letting
$$c(\alpha) = 
 \begin{cases} 
   0 & \text{if there is } \beta_1 \colon  \str B_1\rightarrow \str C\text{ such that }\alpha = \beta_1\circ\alpha_1 \\
   1       & \text{otherwise }.
  \end{cases}$$

For an illustration, see Figure~\ref{fig:ramseyamalg}.

But then, as $\str C \longrightarrow (\str B)^\str A_2$, there is an embedding $\beta \colon  \str B \rightarrow \str C$ such that $c|_{\beta(\str B)}$ is constant. But there are at least two embeddings of $\str A$ into $\beta$ --- one is given by $\alpha_1$ and the other is given by $\alpha_2$. And $\alpha_1$ can be extended to an embedding of $\str B_1$, while $\alpha_2$ can be extended to an embedding of $\str B_2$, hence they got different colours, which is a contradiction.
\end{proof}

\subsection{Expansions}


By the Ne\v set\v ril-R\"odl theorem, the class $\overrightarrow{\mathrm{Gra}}$ of all ordered finite graphs is Ramsey, while one can show that the class $\mathrm{Gra}$ of all finite graphs is not Ramsey. To be able to talk about this difference formally and in general, we need the following definitions.

\begin{definition}[Expansion and reduct]
Let $L$ be a language, let $L^+$ be another language such that $L\subseteq L^+$ (i.e. $L^+$ contains all symbols that $L$ contains and they have the same arities). Then we call $L^+$ an \emph{expansion} of $L$ and we call $L$ a \emph{reduct} of $L^+$.

Let $\str M$ be an $L$-structure and let $\str M^+$ be an $L^+$-structure such that $\str M^+|_L = \str M$ (by this we mean that $\str M$ and $\str M^+$ have the same sets of vertices and the interpretations of symbols from $L$ are exactly the same in both structures). Then we call $\str M^+$ an \emph{expansion} of $\str M$ and we call $\str M$ a \emph{reduct} of $\str M^+$.

If $\mathcal C$ is an amalgamation class of finite $L$-structures, we say that $\mathcal C^+$, an amalgamation class of finite $L^+$-structures is its \emph{expansion} if $\mathcal C = \Age(\str M)$, $\mathcal C^+ = \Age(\str M^+)$ and $\str M^+$ is an expansion of $\str M$.
\end{definition}
\begin{remark*}
In model theory reduct and expansion often mean something more general, but for our purposes this definition is sufficient.
\end{remark*}

Historically, expansions are often called \emph{lifts} in the Ramsey-theoretic context and reducts are called \emph{shadows}. We say that a class has a Ramsey expansion if it has an expansion which is Ramsey.

So far we have only added all linear orders, which is clearly a special expansion (and corresponds to adding the dense linear order with no endpoints to the \Fraisse{} limit which is independent from the rest of the relations). But there are examples where it is not enough.

To sum up, we know that every Ramsey class (with the joint embedding property) has the amalgamation property. Amalgamation classes (of finite structures in a countable language) correspond to homogeneous structures, their \Fraisse{} limits. And as we have seen, by adding some more structure on top of a homogeneous structure and looking at the age, one can get a Ramsey class.

In 2005, Ne\v set\v ril~\cite{Nevsetril2005} started the classification programme of Ramsey classes --- the counterpart of the Lachlan-Cherlin classification programme of homogeneous structures. Its goal is to classify all possible Ramsey classes, a goal quite ambitious, but in some cases achievable; the classification programme of homogeneous structures offers lists of possible Ramsey classes, or rather base classes for expansions. This thesis is a contribution to the Ne\v set\v ril classification programme.

Having read this far into the historical introduction, the reader has probably already asked themselves: Does every amalgamation class have a Ramsey expansion?

The answer to this question is positive, but by cheating: One can add infinitely many unary predicates and let each vertex have its own predicate. Then every structure has at most one embedding to any other and the Ramsey question becomes trivial. There are several ways how to fix the question:
\begin{question}[Bodirsky-Pinsker-Tsankov, 2011~\cite{Bodirsky2011a}]
Does every amalgamation class in a finite language has a Ramsey expansion in a finite language?
\end{question}
This question still remains open. Another possible fix (by Nguyen Van Th\'e) is motivated by topological dynamics and asks whether every amalgamation class with an $\omega$-categorical \Fraisse{} limit has a \textit{precompact} Ramsey expansion. While it is important for the area, it is not necessary for this thesis and we will omit the details. We just note that very recently Evans, Hubi\v cka and Ne\v set\v ril~\cite{Evans2} answered this question negatively.

However, topological dynamics and Nguyen Van Th\'e's results also give a notion of what a ``good'' (or a ``minimal'') expansion is.

\begin{definition}[Expansion property~\cite{The2013}]
Let $\mathcal C$ be a class of finite structures and let $\mathcal C^+$ be its expansion. We say that $\mathcal C^+$ has the \emph{expansion property} (with respect to $\mathcal C$) if for every $\str B\in \mathcal C$ there is $\str C\in \mathcal C$ such that for every $\str B^+\in \mathcal C^+$ and for every $\str C^+\in \mathcal C^+$ such that $\str B^+$ is an expansion of $\str B$ and $\str C^+$ is an expansion of $\str C$ there is an embedding $\str B^+ \rightarrow \str C^+$.
\end{definition}
An expansion has the expansion property if for every small structure $\str B$ in the non-expanded class there is a large structure $\str C$ in the non-expanded class such that every expansion of $\str C$ contains every expansion of $\str B$. The expansion property is a generalization of the ordering property studied by Ne\v set\v ril and R\"odl in the 70's and 80's~\cite{Nesetril1975} and it turns out that it expresses well what an intuitively ``good'' expansion is.

Nguyen Van Th\'e also proves that under certain assumptions there is, up to bi-definability, only one Ramsey expansion with the expansion property. And this is in some sense the best one. It is worth noting that Kechris, Pestov and Todor\v cevi\' c~\cite{Kechris2005} proved this for the special case when the expansion is all linear orders (i.e. the expansion property is the ordering property).

%% file: catalogue.tex
\chapter{Cherlin's catalogue of metrically homogeneous graphs}\label{ch:catalogue}
We are going to study certain classes of finite metric spaces of diameter $\delta$ (that is, metric spaces with distances in the set $\{0, 1, \ldots, \delta\}$) for some positive integer $\delta \geq 3$, i.e. \emph{$\delta$-valued metric spaces}. But usually, we will just call them metric spaces and $\delta$ will be clear from the context. There are two equivalent ways how to look at these metric spaces: One can consider them to be the tuple $(V, d)$, where $V$ is the vertex set and $d\colon  V^2\rightarrow \{0, \ldots, \delta\}$ is the \emph{metric} satisfying the axioms of metric spaces, namely
\begin{enumerate}
\item $d(x,y) = 0$ if and only if $x = y$;
\item $d(x,y) = d(y,x)$; and
\item $d(x,y)+d(y,z)\geq d(x,z)$, the \emph{triangle inequality}.
\end{enumerate}

Equivalently, one can consider the metric space $(V,d)$ to be the relational structure $(V, \rel{}{1}, \ldots, \rel{}{\delta})$, where $(a,b)\in \rel{}{i}$ if and only if $d(a,b) = i$.

Such a structure is essentially a complete undirected graph with edges labelled by integers from $\{1,\ldots,\delta\}$ such that it contains no non-metric triangle (a triangle with labels $a,b,c$ with $a>b+c$) as a substructure. We do not explicitly represent that $d(a,a) = 0$. We shall freely switch between these three interpretations without explicit notice. And thanks to these different points of view, we can now use all the (model-theoretic) notions from the earlier chapters also for metric spaces.

The tuple $(\{x,y,z\}, d)$ with $d(x,y) = d(y,x) = a$, $d(y,z) = d(z,y) = b$ and $d(x,z) = d(z,x) = c$ and $d(x,x)=d(y,y)=d(z,z)$ is called a \emph{triangle}. We will abuse notation and refer as \emph{triangle} to both the three vertices $u,v,w$ and the distances $a,b,c$, but we will try to use the letters $a$, $b$ and $c$ when talking about distances and $u$, $v$ and $w$ when talking about vertices. Finally, to save some visual comma-noise, we often omit the commas and talk about triangles $abc$ or $uvw$.

We say that a $\delta$-valued metric space $\str M = (M, d)$ \emph{omits the triangle $abc$} if there are no $x,y,z\in M$ with $d(x,y) = a$, $d(y,z) = b$ and $d(x,z) = c$. Equivalently, when viewed as a relational structure, there is no embedding of the triangle $abc$ into $\str M$. Let $\mathcal F$ be a class of $\delta$-valued triangles. We say that a class $\mathcal C$ of $\delta$-valued metric spaces is given by forbidding $\mathcal F$ if $\mathcal C$ consists precisely of those $\delta$-valued metric spaces which omit all triangles from $\mathcal F$.

A \emph{metrically homogeneous graph of diameter $\delta$} is a connected countable graph $\str G = (V, E)$ such that the structure $\str M = (V, d)$ is a homogeneous metric space of diameter $\delta$.\footnote{If we look at metric spaces as at relational structures in the language $\{\rel{}{i}\}_{i=1}^\delta$ and identify the edge relation $E$ with $\rel{}{1}$, then the metric space $\str M$ is a \emph{homogenization} of $\str G$, that is, a (minimal) expansion which is homogeneous.} Here $d\colon  V^2\rightarrow \mathbb \{0, \ldots, \delta\}$ assigns each pair of vertices the length (number of edges) of the shortest path connecting them.

We can now give the list of the metric spaces corresponding to the \emph{primitive 3-constrained metrically homogeneous graphs} by means of forbidden triangles. In doing that, we restrict definitions and theorems from~\cite{Cherlin2013}. (Primitive means that there are no \emph{algebraic closures} and no \emph{definable equivalence relations}, 3-constrained means that they are determined by constraints on \emph{3-types}, see~\cite{Cherlin2013}.) The other classes from Cherlin's catalogue are either very simple, or are extremal variants of the primitive 3-constrained classes and they are handled similarly to the primitive 3-constrained ones. For more about the other classes, see Chapter~\ref{ch:conclusion}.

 To describe these classes, five numerical parameters are needed: $(\delta, K_1, K_2,\allowbreak C_0, C_1)$. As we have already seen, $\delta$ is the diameter of the metric spaces. The other parameters are used to describe several types of forbidden triangles. There will be two ways of restricting which sequences are of interest. The \emph{acceptability conditions} give rough restrictions on the parameters and also ensure that no class can be described by more than one sequence of parameters.

\begin{definition}[Acceptable numerical parameters]
A sequence of positive integers $(\delta,K_1,K_2,C_0,C_1)$ is \emph{acceptable} if it satisfies the following conditions:
\begin{itemize}
  \setlength\itemsep{0em}
  \item $3\leq \delta < \infty$;
  \item $1\leq K_1\leq K_2\leq \delta$; and
  \item $2\delta+2\leq C_0,C_1\leq 3\delta+2$. Furthermore $C_0$ is even and $C_1$ is odd.
\end{itemize}
\end{definition}

Now we can describe the parametrized classes $\Aclass$ in terms of forbidden triangles:

\begin{definition}[Triangle constraints]
\label{defn:numerical}
Given acceptable parameters $\delta$, $K_1$, $K_2$, $C_0$ and $C_1$ we consider the class
$\Aclass$ of all finite $\delta$-valued metric spaces $\str{M}=(M,d)$ such that for every three distinct vertices $u\neq v\neq w\in M$ the following are true, where $p=d(u,v)+d(u,w)+d(v,w)$ is the perimeter of the triangle $u,v,w$ and $m=\min\{d(u,v),\allowbreak d(u,w),\allowbreak d(v,w)\}$ is the length of the shortest edge of $u,v,w$.
\begin{description}
 \setlength\itemsep{0em}
 \item[$K_1$-bound] If $p$ is odd then $2K_1 < p$;
 \item[$K_2$-bound] if $p$ is odd then $p < 2K_2 + 2m$;
 \item[$C_1$-bound] if $p$ is odd then $p<C_1$; and
 \item[$C_0$-bound] if $p$ is even then $p<C_0$.
\end{description}
\end{definition}
Intuitively, the parameter $K_1$ forbids all odd cycles shorter than $2K_1+1$, while $K_2$ ensures that the difference in length between even- and odd-distance paths connecting any
pair of vertices is less than $2K_2+1$. The parameters $C_0$ and $C_1$
forbid long even and odd cycles respectively.

Not every combination of numerical parameters leads to a strong amalgamation class (Definitions~\ref{defn:amalgamation} and~\ref{defn:strongamalgamation}). Those that do are characterised by the following theorem

\begin{theorem}[Cherlin's Admissibility Theorem~\cite{Cherlin2013}, simplified for primitive 3-constrained cases]
\label{thm:admissible}
Let $(\delta,K_1,K_2,C_0,C_1)$ be an acceptable sequence of parameters (in particular, $\delta\geq 3$). Then
the associated class $\mathcal A^\delta_{K_1,K_2,C_0,C_1}$ is a strong amalgamation class if
one of the following three groups of conditions is satisfied, where we write $C$ for $\min(C_0,C_1)$
and $C'$ for $\max(C_0,C_1)$:
\begin{enumerate}[label=(\Roman*)]
  \setcounter{enumi}{1}

\setlength\itemsep{0em}
\item\label{II} $C\leq 2\delta+K_1$, and
\begin{itemize}
 \setlength\itemsep{0em}
 \item $C=2K_1+2K_2+1$;
 \item $K_1+K_2\geq \delta$;
 \item $K_1+2K_2\leq 2\delta-1$, and:
\end{itemize}
\begin{enumerate}[label=(II\Alph*)]
\setlength\itemsep{0em}
\item\label{IIa} $C'=C+1$, or
\item\label{IIb} $C'>C+1, K_1=K_2$, and $3K_2=2\delta-1$.
\end{enumerate}
\item\label{III} $C\geq 2\delta+K_1+1$, and:
\begin{itemize}
 \setlength\itemsep{0em}
 \item $K_1+2K_2\geq 2\delta-1$ and $3K_2\geq 2\delta$;
 \item If $K_1+2K_2=2\delta-1$ then $C\geq 2\delta+K_1+2$;
 \item If $C'>C+1$ then $C\geq 2\delta+K_2$.
\end{itemize}
\end{enumerate}
\end{theorem}
A sequence of parameters $(\delta,K_1,K_2,C_0,C_1)$ is called {\em admissible} if and only if it satisfies one of the sets of conditions in Theorem~\ref{thm:admissible}. Note that there is no Case I, as it corresponds to the bipartite spaces in Cherlin's catalogue which are not primitive and hence not considered in this thesis.

\begin{example}
All admissible parameters with $\delta=3$ are listed in Table~\ref{tab:delta3}.
\end{example}

\begin{table}[h]
\centering 
\begin{tabular}{|rrrr|r|r|l|}
\hline
 $K_1$   & $K_2$&$C_0$& $C_1$& $M$ & Case  & Structure \\ \hline
 1       & 2 & 10 &  9 & 2 & \ref{III} & No $\delta\delta\delta$, $1\delta\delta$ triangles\\ 
 1       & 2 & 10 & 11 & 2 & \ref{III} & No $1\delta\delta$ triangles\\ 
 1       & 3 &  8 &  9 & 2 & \ref{III} & No 5-anticycle ($\delta\delta\delta\delta\delta$) \\ 
 1       & 3 & 10 & 9 & 2 & \ref{III} & No $\delta\delta\delta$ triangles\\ 
 1       & 3 & 10 & 11 & $2, 3$ & \ref{III} & All metric spaces \\ 
 2       & 2 & 10 & 9 & 2 & \ref{III}  & No $\delta\delta\delta$, $1\delta\delta$, $111$\\ 
 2       & 2 & 10 & 11 & 2 & \ref{III} & No $1\delta\delta$, $111$ triangles\\ 
 2       & 3 & 10 & 9 & 2 & \ref{III} & No $\delta\delta\delta$, $111$ triangles\\ 
 2       & 3 & 10 & 11 & $2, 3$ & \ref{III} & No $111$ triangles \\ 
 3       & 3 & 10 & 11 & 3 & \ref{III} & No 5-cycle \\ \hline
\end{tabular}
\caption{All admissible parameters for $\delta=3$. The second column lists the possible choices for magic distances (see Definition~\ref{defn:magicdistance}). The ``Structure'' column tries to describe the defining structural property; it does not list all the forbidden substructures.}
\label{tab:delta3}
\end{table}

%% file: preliminaries.tex
\chapter{Preliminaries}\label{ch:preliminaries}
We are going to work with \textit{incomplete} metric spaces (in the sense that not all distances are defined). Recall from previous chapters that a structure $\str G = (V, E, d)$ is a \emph{$\delta$-edge-labelled graph} if $(V,E)$ is an undirected graph without loops and $d\colon E\rightarrow\{1,\ldots,\delta\}$ is a (distance) function.

Clearly the edge relation is redundant as it can be inferred from the domain of $d$. We can also treat $d$ as a partial function from $V^2$ to $\{0, 1, \ldots, \delta\}$, which is symmetric, zero for $d(x,x)$ an undefined whenever $xy\notin E$. The last possible way of looking at a $\delta$-edge-labelled graph is as on a relational structure with relations $\rel{}{1}, \ldots, \rel{}{\delta}$, all of them binary and symmetric, where each pair of vertices is in at most one relation. We will use all these perspectives and will switch between them implicitly.

By $\mathcal G^\delta$ we will denote the class of all finite $\delta$-edge-labelled graphs.

\section{The Hubi\v cka-Ne\v set\v ril theorem}
To prove the Ramsey property, we are going to employ a deep theorem by Hubi\v cka and Ne\v set\v ril~\cite{Hubicka2016}, which is going to do all the heavy lifting and complex constructions for us, we ``only'' check the conditions under which their theorem can be applied. We will state a weaker variant of the results which is simpler, yet sufficient for our purposes.

Recall that a structure is irreducible if every pair of vertices is in some relation.
\begin{definition}
Let $L$ be a relational language and $\str A, \str B$ arbitrary $L$-structures. We say that a map $f\colon  A\rightarrow B$ is a \emph{homomorphism} if for every relation $\rel{}{} \in L$ it holds that whenever $(x_1, \ldots, x_{\arity{}}) \in \rel{A}{}$, then also $(f(x_1), \ldots, f(x_{\arity{}})) \in \rel{B}{}$. We write $f\colon \str A\rightarrow \str B$ for a homomorphism $f$ to emphasize that it respects the structure.

A homomorphism $f\colon \str A\rightarrow \str B$ is a \emph{homomorphism-embedding} if for every irreducible $\str C\subseteq \str A$ it holds that $f$ restricted to $\str C$ is an embedding.

For a class of $L$-structures $\mathcal F$, by $\Forb(\mathcal F)$ we denote the class of all finite $L$-structures $\str A$ such that there is no $\str F\in \mathcal F$ with a homomorphism $\str F\rightarrow \str A$.
\end{definition}

\begin{definition}[Completion]\label{defn:completion}
Let $L$ be a language and $\str C$ be a structure. An irreducible structure $\str C'$ is a \emph{(strong) completion} of $\str C$ if there is an injective homomorphism-embedding $f\colon \str C\rightarrow \str C'$.
\end{definition}

\begin{definition}[Locally finite subclass~\cite{Hubicka2016}]
Let $L$ be a language, $\mathcal R$ be a class of finite irreducible structures and $\mathcal K\subseteq \mathcal R$ a subclass of $\mathcal R$. We say that $\mathcal K$ is a \emph{locally finite subclass of $\mathcal R$} if for every $\str C_0 \in \mathcal R$ there exists an integer $n=n(\str C_0)$ such that for every finite $L$-structure $\str C$ there exists $\str C' \in \mathcal K$ which is a completion of $\str C$ provided that:
\begin{enumerate}
\item There exists a homomorphism-embedding from $\str C$ to $\str C_0$; and
\item for every substructure $\str S\subseteq \str C$ such that $\str S$ has at most $n$ vertices there exists $\str S'\in \mathcal K$ which is a completion of $\str S$.
\end{enumerate}
\end{definition}

Now we can state the main result of~\cite{Hubicka2016}.
\begin{theorem}[Hubi\v cka-Ne\v set\v ril~\cite{Hubicka2016}]\label{thm:hn}
Let $L$ be a language and $\mathcal R$ be a class of finite irreducible $L$-structures which has the Ramsey property. Let $\mathcal K\subseteq \mathcal R$ be a locally finite subclass of $\mathcal R$ which has the strong amalgamation property and is hereditary (if $\str B\in \mathcal K$ and $\str A\subseteq \str B$, then $\str A\in \mathcal K$). Then $\mathcal K$ is Ramsey.
\end{theorem}

In Chapter~\ref{ch:metricspaces} we give an easy example of an application of Theorem~\ref{thm:hn} and then in the rest of the thesis we use it for Cherlin's metric spaces: In Chapter~\ref{ch:magiccompletion}, we prove that for every class $\Aclass$ there is a finite family $\mathcal F$ of $\delta$-edge-labelled cycles such that a finite $\delta$-edge-labelled graph $\str G$ has a completion in $\Aclass$ if and only if $\str G \in \Forb(\mathcal F)$. The finiteness of $\mathcal F$ in particular ensures that $\Aclass$ is a locally finite subclass of all $\delta$-edge-labelled graphs. To obtain the Ramsey property, we need to add linear orders, but as they are independent from the metric structure, the argument will be essentially the same.

\section{EPPA}
There is another combinatorial property for which there turns out to be a theorem similar to Theorem~\ref{thm:hn} and which we, hence, get almost for free.\footnote{The paper~\cite{Aranda2017} is one of the first papers where the Ramsey property and EPPA are proved at the same time.}

Here we only prove a weak variant of what is proved in~\cite{Aranda2017} as the goal is not to present a lot of necessary definitions, but to advocate the usefulness of the study of completions to various classes. For an overview of the whole area consult the PhD thesis of Daoud Siniora~\cite{Siniora2}.

\begin{definition}[EPPA]\label{defn:eppa}
Let $L$ be a language and $\mathcal C$ be a class of finite $L$-structures. We say that $\mathcal C$ has the \emph{extension property for partial automorphisms} (\emph{EPPA}) if for every $\str A\in \mathcal C$ there exists $\str B\in \mathcal C$ and an embedding $\alpha\colon \str A\rightarrow \str B$ such that for every partial automorphism $f\colon  \str A\rightarrow \str A$ (isomorphism of substructures of $\str A$) there exists an automorphism $g$ of $\str B$ such that $g$ extends $f$, or formally $\alpha\circ f \subseteq g\circ \alpha$.
\end{definition}
EPPA is sometimes also called the \emph{Hrushovski property}, because Hrushovski was the first to prove that the class of all finite graphs has EPPA~\cite{hrushovski1992}.

A distant analogue of the Hubi\v cka-Ne\v set\v ril theorem for EPPA is the following result by Herwig and Lascar.
\begin{theorem}[Herwig-Lascar~\cite{herwig2000}]\label{thm:herwiglascar}
Let $L$ be a language, $\mathcal F$ be a finite family of $L$-structures and $\str{A}\in \Forb(\mathcal F)$ a finite $L$-structure. If there exists an $L$-structure $\str{M}$ containing $\str{A}$ such that every partial automorphism of $\str A$ extends to an automorphism of $\str{M}$ and moreover there is no $\str{F}\in \mathcal F$ with a homomorphism $\str{F}\to\str{M}$, then there exists a finite structure $\str{B}\in \Forb(\mathcal F)$ containing $\str A$ such that every partial automorphism of $\str A$ extends to an automorphism of $\str B$.
\end{theorem}
If $\str A\in \mathcal K\subseteq \Forb(\mathcal F)$ where $\mathcal K$ is a \Fraisse{} class, then one can simply take $\str M$ as the \Fraisse{} limit of $\mathcal K$.

In this thesis, we are dealing with complete structures, so the Herwig-Lascar theorem is not sufficient by itself, one needs a very easy corollary.

\begin{definition}[Automorphism-preserving completion]
Let $L$ be a language, $\mathcal K$ be a class of finite irreducible $L$-structures and $\str C$ be an $L$-structure. We say that $\str C$ has an \emph{automorphism-preserving completion} in $\mathcal K$ if there is $\str C'\in \mathcal K$ which is a strong completion of $\str C$ (i.e. there exists an injective homomorphism-embedding $f\colon \str C\rightarrow \str C'$, cf. Definition~\ref{defn:completion}) and further if $\alpha\colon \str C\rightarrow \str C$ is an automorphism of $\str C$, then there is an automorphism $\beta\colon \str C'\rightarrow \str C'$ such that $f\circ \alpha = \beta\circ f$ (every automorphism of $\str C$ can be extended to an automorphism of $\str C'$).
\end{definition}

\begin{corollary}\label{cor:herwiglascar}
Let $L$ be a language and $\mathcal K$ be a \Fraisse{} class of finite irreducible $L$-structures. If there exists a finite family $\mathcal F$ of finite $L$-structures such that $\mathcal K\subseteq \Forb(\mathcal F)$ and every $\str C\in \Forb(\mathcal F)$ has an automorphism-preserving completion in $\mathcal K$, then $\mathcal K$ has EPPA.
\end{corollary}
\begin{proof}
Take an arbitrary $\str A\in \mathcal K$. As $\mathcal K$ is a \Fraisse{} class and $\mathcal K\subseteq \Forb(\mathcal F)$, both conditions of Theorem~\ref{thm:herwiglascar} are satisfied and hence there is a structure $\str B\in \Forb(\mathcal F)$ which is an \emph{EPPA-witness} for $\str A$ in $\Forb(\mathcal F)$. Now it is enough to take the automorphism-preserving completion of $\str B$ in $\mathcal K$.
\end{proof}

%% file: metricspaces.tex
\chapter{Metric spaces}\label{ch:metricspaces}
In~2007 Ne\v set\v ril proved that the class of all finite ordered metric spaces has the Ramsey property~\cite{Nevsetvril2007}, while Solecki~\cite{solecki2005} and Vershik~\cite{vershik2008} independently proved that the class of all finite metric spaces has EPPA.\footnote{It is worth mentioning that Ma\v sulovi\'c~\cite{Masulovic2017metric} gave a simpler proof of Ne\v set\v ril's theorem by a reduction to the Graham-Rothschild theorem.}

As a warm-up for the next chapter, we shall use the modern techniques to prove that the class of all ordered finite $\delta$-valued metric spaces is Ramsey. Almost for free (using Corollary~\ref{cor:herwiglascar}) we also get EPPA. The subsequent chapters generalise the approach introduced here, but are technically much more challenging.

Let $\delta\geq 2$ be an integer. By $\mathcal M^\delta$ we denote the class of all finite $\delta$-valued metric spaces. We will look at the structures both as on $\delta$-edge-labelled graphs and as tuples $(V, d)$, where $d\colon V^2\rightarrow \{0,1,2,\ldots,\delta\}$ is the metric. Finally by $\overrightarrow{\mathcal M^\delta}$ we mean the class of all ordered finite $\delta$-valued metric spaces; for every $(V, d)\in \mathcal M^\delta$ there are all $|V|!$ of its expansions in $\overrightarrow{\mathcal M^\delta}$. Take it as a fact (but it also follows from the following theorem) that $\mathcal M^\delta$ is a \Fraisse{} class with the strong amalgamation property.

For $\delta\geq 2$, let $\mathcal F^\delta\subset \mathcal G^\delta$ be the class of all \emph{non-metric cycles}, that is, $\delta$-edge-labelled cycles with distances $\ell, a_1, \ldots, a_k$ such that $\ell > \sum_i a_i$.

We first state the main result of this short chapter and use it to prove two corollaries.
\begin{theorem}\label{thm:metriccompletion}
Fix an integer $\delta\geq 2$. Then $\str G\in \mathcal G^\delta$ has a completion in $\mathcal M^\delta$ if and only if $\str G\in \Forb(\mathcal F^\delta)$.

Furthermore, for every $\str C = (C, d)\in \Forb(\mathcal F^\delta)$ there exists $\str C' = (C, d') \in \mathcal M^\delta$ such that $\str C'$ is an automorphism-preserving completion of $\str C$ and for every $\str C'' = (C, d'')\in \mathcal M^\delta$ which is a completion of $\str C$ and for every $x,y\in C$ it holds that $d'(x,y) \geq d''(x,y)$.
\end{theorem}

Note that $\mathcal F^\delta$ is finite for every $\delta$ (because every $\str F\in \mathcal F^\delta$ has at most $\delta$ vertices). With this in mind, Corollary~\ref{cor:herwiglascar} gives us immediately:
\begin{corollary}
$\mathcal M^\delta$ has EPPA for every $\delta$.
\end{corollary}

To prove the Ramsey property, we need to do a little more work.
\begin{corollary}\label{cor:metricareramsey}
$\overrightarrow{\mathcal M^\delta}$ has the Ramsey property for every $\delta$.
\end{corollary}
\begin{proof}
Assume that $\delta$ is fixed. Let $\overrightarrow{\mathcal G^\delta}$ denote the class of all finite linearly ordered $\delta$-edge-labelled graphs. By the Ne\v set\v ril-R\"odl theorem $\overrightarrow{\mathcal G^\delta}$ is Ramsey: Let $L$ be the binary language $L=(\rel{}{1}, \ldots, \rel{}{\delta})$. Then $\overrightarrow{\mathcal G^\delta}$ is the class of all linearly ordered $L$-structures such that every pair of vertices is in at most one $L$-relation. A pair of vertices which is in more than one relation is clearly irreducible, so one can forbid all such pairs.

To prove that $\overrightarrow{\mathcal M^\delta}$ is Ramsey, we will use Theorem~\ref{thm:hn} where $\overrightarrow{\mathcal G^\delta}$ plays the role of $\mathcal R$ (the order is linear, so all structures in $\overrightarrow{\mathcal G^\delta}$ are irreducible) and $\overrightarrow{\mathcal M^\delta}$ plays the role of $\mathcal K$. It remains to verify local finiteness.

Let $n$ be the maximum number of vertices of any member of $\mathcal F^\delta$. Fix a $\str C_0\in \overrightarrow{\mathcal M^\delta}$ and let $\str C$ be a structure in the language $(\leq, \rel{}{1}, \ldots, \rel{}{\delta})$ with a homomorphism-embedding $h\colon \str C\rightarrow\str C_0$.

The existence of $h$ ensures that every pair of vertices of $\str C$ is in at most one $\rel{}{i}$ relation and that the relation $\leq_\str{C}$ is acyclic. Because in $\overrightarrow{\mathcal M^\delta}$ we allow all linear orders, we can complete the order arbitrarily to a linear order and forget about it. It remains to fill in the missing distances. Let $\str G\in \mathcal G^\delta$ be the reduct of $\str C$ which one gets by forgetting the order. By the assumption, every substructure of $\str C$ on at most $n=\delta$ vertices has a completion in $\overrightarrow{\mathcal M^\delta}$, which implies that every substructure of $\str G$ on at most $n=\delta$ vertices has a completion in $\mathcal M^\delta$. But this means that $\str G\in \Forb(\mathcal F^\delta)$ and hence there is a completion of $\str G$ in $\mathcal M^\delta$ which combined with the linear order gives a completion of $\str C$ in $\overrightarrow{\mathcal M^\delta}$.
\end{proof}

\section{The shortest path completion}
Now it suffices to prove Theorem~\ref{thm:metriccompletion}. We will do it by finding an explicit completion procedure and proving that it has all the desired properties. This folklore construction was used in Ne\v set\v ril's paper on the Urysohn space~\cite{Nevsetvril2007} as well as for example Solecki's~\cite{solecki2005} and Vershik's~\cite{vershik2008} results on EPPA. To the author's best knowledge, the name \emph{shorest path completion} was first used in the paper of Hubi\v cka and Ne\v set\v ril~\cite{Hubicka2016}.

Let $\str G = (V, E, d)$ be a $\delta$-edge-labelled graph. By a walk $W$ (from $v_1$ to $v_k$) in $\str G$ we mean a sequence of vertices $v_1, v_2, \ldots, v_k \in V$ such that $v_iv_{i+1}\in E$ for every $1\leq i \leq k-1$. We further define $\|W\| = \sum_{i=1}^{k-1} d(v_i, v_i+1)$ and call it the \emph{length of $W$} (in $\str G$). A \emph{path} $P$ is a walk such that the vertices are distinct.

\begin{definition}[Shortest path completion]\label{def:spp}
Let $\str G = (V, E, d)$ be a $\delta$-edge-labelled graph. Define $d' \colon  V^2 \rightarrow \{0, \ldots, \delta\}$ as
$$d'(x,y) = \min\left(\delta, \min_{\substack{P \text{ path in }\str G\\\text{from }x\text{ to }y}} \|P\|\right).$$
We call $(V, d')$ the \emph{shortest path completion} of $\str G$.
\end{definition}

Note that if $\str G$ is disconnected and $x,y$ are in different components then $d'(x,y) = \delta$ and also if $xy\in E$, then the edge $xy$ is also a path from $x$ to $y$. The \emph{completion} part of the name is a little bit misleading --- from the definition it is not even clear that $d'|_E = d$. In fact, we shall see that if $(V,d')$ has no completion in $\mathcal M^\delta$, then $d'|_E \neq d$, but otherwise  whenever $\str G$ has a completion in $\mathcal M^\delta$ then $(V, d')$ really is a completion of $\str G$ in $\mathcal M^\delta$.

\begin{prop}\label{prop:shortestpath}
Fix a $\delta\geq 2$. Let $\str G = (V, E, d)$ be a $\delta$-edge-labelled graph and $(V, d')$ be its shortest path completion. Then the following hold:
\begin{enumerate}[label=(\arabic*)]
\item \label{prop:shortestpath:0} If $\str G \in \mathcal F^\delta$ is a non-metric cycle, then $\str G$ has no completion in $\mathcal M^\delta$;
\item \label{prop:shortestpath:0.5} $\str G$ contains a path from $x$ to $y$ of length at most $a$ if and only if it contains a walk from $x$ to $y$ of length at most $a$;
\item \label{prop:shortestpath:1} $(V, d') \in \mathcal M^\delta$;
\item \label{prop:shortestpath:1.5} There exists $\str F\in \mathcal F^\delta$ with a homomorphism $f\colon \str F\rightarrow \str G$ if and only if there exists $\str F'\in \mathcal F^\delta$ with an injective homomorphism $g\colon \str F'\rightarrow \str G$.
\item \label{prop:shortestpath:2} $d'|_E = d$ if and only if $\str G \in \Forb(\mathcal F^\delta)$;
\item \label{prop:shortestpath:3} if $(V, d'')\in \mathcal M^\delta$ is a completion of $\str G$, then for every $x,y\in V$ it holds that $d'(x,y)\geq d''(x,y)$;
\item \label{prop:shortestpath:4} $\str G$ has a completion in $\mathcal M^\delta$ if and only if $\str G \in \Forb(\mathcal F^\delta)$; and
\item \label{prop:shortestpath:5} if $\str G \in \Forb(\mathcal F^\delta)$ then the shortest path completion preserves automorphisms.
\end{enumerate}
\end{prop}

\begin{proof}
We will prove the claims one by one.

\paragraph{\ref{prop:shortestpath:0}} Enumerate the vertices as $V = \{v_1, \ldots, v_k\}$ such that $v_iv_{i+1}\in E$ (we identify $v_{k+1} = v_1$) and $d(v_1,v_k) > \sum_{i=1}^{k-1} d(v_i, v_{i+1})$. For a contradiction suppose that the statement is not true, that is, $\str G$ has a completion $\str G' = (V, d') \in \mathcal M^\delta$ and among all such cycles, take $\str G$ to have the smallest number of vertices $k$. Clearly $k > 3$, because if $k = 3$ then $\str G$ is a non-metric triangle which is not in $\mathcal M^\delta$. Now take $\str G'$ and look at the triangle $v_1, v_2, v_3$. As $\str G'$ is a completion of $\str G$, we have $d'(v_1, v_2) = d(v_1, v_2)$ and $d'(v_2, v_3) = d(v_2, v_3)$, and because $\str G'\in \mathcal M^\delta$, it holds that $d'(v_1, v_3) \leq d'(v_1, v_2) + d'(v_2, v_3)$. But this is a contradiction with the minimality of $k$, as $\str G'$ is also the completion of the cycle $\str G'' = (\{v_1, v_3, \ldots, v_k\}, d'')$ where $d''(v_1, v_3) = d'(v_1, v_3)$ and $d''$ is equal to $d$ elsewhere, which is a non-metric cycle with fewer vertices than $\str G$.

\paragraph{\ref{prop:shortestpath:0.5}} A path is a walk, so we only need to prove the other implication. A walk can be transformed to a path by replacing all subwalks $v_i, \ldots, v_j$ where $v_i = v_j$ with just the vertex $v_i$. As all distances are non-negative, this procedure doesn't increase the length of the walk.
\smallskip

\paragraph{\ref{prop:shortestpath:1}} Suppose for a contradiction that $(V, d') \notin \mathcal M^\delta$. Clearly $d'(x,x) = 0$ for all $x$ and $d'$ is symmetric. Hence there are $x,y,z\in V$ such that $d'(x,y) > d'(y,z) + d'(z,x)$. Thus, in particular, $d'(y,z), d'(z,x) < \delta$. By definition of $d'$ there is a path $P_1$ in $\str G$ from $y$ to $z$ such that $\|P_1\| = d'(y,z)$ and a path $P_2$ from $z$ to $x$ such that $\|P_2\| = d'(z,x)$. But this means that the concatenation of $P_1$ and $P_2$ is a walk $P$ from $x$ to $y$ going through $z$ of length at $d'(y,z) + d'(z,x)$, which is together with~\ref{prop:shortestpath:0.5} a contradiction with the definition of $d'(x,y)$.
\smallskip

\paragraph{\ref{prop:shortestpath:1.5}} Follows by~\ref{prop:shortestpath:0.5} as a homomorphic image of a cycle $\str F$ with edges $\ell, a_1, \ldots, a_k$ such that $\ell > \sum a_i$ is just the long edge $\ell$ plus a walk consisting of the edges $a_1, \ldots, a_k$. As $a_i < \ell$ for all $a_i$, after shortcutting the walk to a path and adding back the edge $\ell$, we get a cycle from $\mathcal F$ with no vertices glued.
\smallskip

\paragraph{\ref{prop:shortestpath:2}} If $\str G \notin \Forb(\mathcal F^\delta)$ then by~\ref{prop:shortestpath:1.5} there is a non-metric cycle as a non-induced subgraph of $\str G$. Let the vertices of this cycle be $v_1, v_2, \ldots, v_k$ with $d(v_1, v_k) > \sum_{i=1}^{k-1} d(v_i, v_{i+1})$. This means that $d'(v_1, v_k) \leq \sum_{i=1}^{k-1} d(v_i, v_{i+1}) < d(v_1, v_k)$.

Now suppose that there are vertices $x,y$ such that $d'(x,y) < d(x,y)$. This means, by the definition of the shortest path completion, that there is a path from $x$ to $y$ of length strictly less than $d(x,y)$, hence it forms a non-metric cycle together with the edge $xy$.
\smallskip

\paragraph{\ref{prop:shortestpath:3}} Take any completion $\str G'' = (V, d'') \in \mathcal M^\delta$ of $\str G$ and look at an arbitrary pair of vertices $x,y$. If $d'(x,y) = \delta$, then the statement holds. Otherwise $d'(x,y) < \delta$ and this means that there is a path from $x$ to $y$ in $\str G$ of length $d'(x,y)$. As this path is also in $\str G''$, then $d''(x,y) \leq d'(x,y)$, because otherwise $\str G''$ would contain a non-metric cycle which would together with~\ref{prop:shortestpath:0} contradict $\str G''\in \mathcal M^\delta$.
\smallskip

\paragraph{\ref{prop:shortestpath:4}} If $\str G \in \Forb(\mathcal F^\delta)$, then the shortest path completion $\str G'\in \mathcal M^\delta$ is a completion of $\str G$ by~\ref{prop:shortestpath:1} and~\ref{prop:shortestpath:2}. If $\str G \notin \Forb(\mathcal F^\delta)$, then by~\ref{prop:shortestpath:1.5} $\str G$ contains a non-metric cycle as a non-induced subgraph and hence by~\ref{prop:shortestpath:0} has no completion in $\mathcal M^\delta$ (if it had one, then the non-metric cycle would have a completion in $\mathcal M^\delta$ as well).
\smallskip

\paragraph{\ref{prop:shortestpath:5}} Let $\alpha\colon \str G\rightarrow \str G$ be an automorphism of $\str G$. We will show that the same map is also an automorphism of $\str G'$. For this it suffices to check for all $x,y$ that $d'(x,y) = d'(\alpha(x), \alpha(y))$. But if there is a path of length $a$ from $x$ to $y$ in $\str G$, then as $\alpha$ is an automorphism there has to be a path of length $a$ from $\alpha(x)$ to $\alpha(y)$ and vice versa. This also implies that there is no path between $x$ an $y$ if and only if there is no path between $\alpha(x)$ and $\alpha(y)$.

It follows that the shortest path connecting $x$ and $y$ has the same length as the shortest path connecting $\alpha(x)$ and $\alpha(y)$, hence $d'(x,y) = d'(\alpha(x), \alpha(y))$.
\end{proof}

Theorem~\ref{thm:metriccompletion} now follows from Proposition~\ref{prop:shortestpath}.

%% file: magiccompletion.tex
\chapter{The magic completion algorithm}\label{ch:magiccompletion}
A natural thing to do is to try the shortest path completion algorithm for the classes $\Aclass$, because when $K_1 = 1$, $K_2 = \delta$ and $C_0,C_1\geq 3\delta+1$ the class $\Aclass$ is precisely $\mathcal M^\delta$. But unfortunately this does not work in general as the shortest path completion tends to produce triangles of long perimeter. To fix this, we now devise a completion which, instead of being the completion where all edges are as long as possible, will essentially be the completion where all edges are as close to some suitable \emph{magic} parameter $M$ as possible. Specifically, we prove:

\begin{theorem} \label{thm:magiccompletion}
Let $(\delta, K_1,K_2,C_0,C_1)$ be an admissible sequence of parameters.
Suppose that $\str{G}=(G,d)\in \mathcal G^\delta$ has a completion into $\Aclass$. Then for every magic parameter $M$ (see Definitions~\ref{defn:magicdistance} and~\ref{defn:magiccompletion}) --- which always exists --- there is a completion $\overbar{\str{G}}=(G,\bar{d}) \in\Aclass$ of $\str G$ such that it is optimal in the following sense:
Let $\str{G}'=(G,d')\in\Aclass$ be an arbitrary completion of $\str{G}$ in  $\Aclass$, then 
for every pair of vertices $u,v\in G$ one of the following holds:
\begin{enumerate}
 \item $d'(u,v) \geq \bar{d}(u,v) \geq M$,
 \item $d'(u,v) \leq \bar{d}(u,v) \leq M$,
 \item the parameters satisfy Case \ref{IIb}, $d'(u,v) \neq M$ and $\bar{d}(u,v) = M-1$.
\end{enumerate}

Finally, every automorphism of $\str{G}$ is also an automorphism of $\overbar{\str{G}}$.
\end{theorem}

Throughout the chapter, we assume that $\delta, K_1,K_2,C_0,C_1$ are fixed admissible parameters as given by Theorem~\ref{thm:admissible}.
Recall that we denote $C=\min(C_0,C_1),$ and $C'=\max(C_0,C_1).$

Recall also that $\Aclass$ was defined in Definition~\ref{defn:numerical} as the class of all metric spaces that satisfy some constraints on its triangles, i.e. its 3-element subspaces. In the following it will often be more convenient to think of $\Aclass$ as the class of metric spaces that do not embed any triangle violating those constraints --- we are going to refer to such triangles as \emph{forbidden triangles} (cf. Chapter~\ref{ch:preliminaries}). We slightly abuse notation and use the term triangle for both triples of vertices $u,v,w$ and for triples of edges $a,b,c$ with $a = d(u,v)$, $b = d(v,w)$, $c = d(u,w)$. By Definition~\ref{defn:numerical} a triangle $a,b,c$ with $a\leq b\leq c$ is forbidden if it satisfies one of the following conditions: 
\begin{description}
\setlength\itemsep{0em}

\item[Non-metric:] $a+b<c$;
\item[$K_1$-bound:] $a+b+c < 2K_1+1$ and $a+b+c$ is odd;
\item[$K_2$-bound:] $b+c\geq 2K_2+a$ and $a+b+c$ is odd and $a\leq b,c$;
\item[$C_1$-bound:] $a+b+c\geq C_1$ and $a+b+c$ is odd;
\item[$C_0$-bound:] $a+b+c\geq C_0$ and $a+b+c$ is even.
\item[$C$-bound:] If $|C_0-C_1| = 1$, the $C_1$-bound and $C_0$-bound can be expressed together as $a+b+c\geq C$, where $C=\min(C_0,C_1)$.
\end{description}
Triangles that are not forbidden will be called \emph{allowed}.

\section{The magic completion algorithm}
\label{sec:algorithm}

Let $\mathcal D=\{1,2,\ldots \delta\}^2$ be a collection of (ordered) pairs. It is more natural to consider unordered pairs, but notationally easier to consider ordered pairs. We will refer to elements of $\mathcal D$ as {\em forks}. If $\str G$ is a $\delta$-edge-labelled graph, we will also refer to its triples triples of vertices $u,v,w$ such that $uv$ and $uw$ are edges and $vw$ is not an edge as \emph{forks}.

Consider the shortest path completion for $\mathcal M^\delta$ from Definition~\ref{def:spp}. There is an alternative formulation of this completion: For a fork $\vec{f}=(a,b)$, define $d^+(\vec{f})=\min(a+b,\delta)$. Proceed in steps numbered from $1$ to $\delta$ and in the $i$-th step look at all (incomplete) forks $\vec{f}$ such that $d^+(\vec{f}) = i$ and define the length of the missing edge to be $i$.

This algorithm proceeds by first adding edges of length 2, then edges of length 3 and so on up to edges of length $\delta$ and has the property that out of all metric completions of a given graph, every edge of the completion yielded by this algorithm is as close to $\delta$ as possible.

It makes sense to ask what happens if, instead of trying to make each edge as close to $\delta$ as possible, one tries to make each edge as close to some parameter $M$ as possible. For $M$ in a certain range, such an algorithm exists. For each fork $\vec{f}=(a,b)$ one can define $d^+(\vec{f}) = a+b$ and $d^-(\vec{f}) = |a-b|$, i.e. the largest and the smallest possible distance that can metrically complete the fork $\vec{f}$. The generalised algorithm will, again in steps, complete $\vec{f}$ by $d^+(\vec{f})$ if $d^+(\vec{f})<M$, by $d^-(\vec{f})$ if $d^-(\vec{f})>M$ and by $M$ otherwise. It turns out that there is a good permutation $\pi$ of $\{1,\ldots,\delta\}$, such that if one adds the distances in the order prescribed by the permutation, this generalised algorithm will produce a correct completion whenever one exists. It is easy to check that the choice $M=\delta$ and $\pi=\text{id}_\delta$ corresponds to the shortest path completion algorithm.

\begin{definition}[Completion algorithm]\label{defn:ftmcompletion}
Given $c\geq 1$, $\mathcal F\subseteq \mathcal D$, and
a $\delta$-edge-labelled graph $\str{G}=(G,d)\in \mathcal G^\delta$, we say that a $\delta$-edge-labelled $\str{G}'=(G,d')$ is the 
{\em $(\mathcal F,c)$-completion} of $\str{G}$ if $d'(u,v)=d(u,v)$ whenever $u,v$ is an edge of $\str{G}$
and $d'(u,v)=c$ if $u,v$ is not an edge of $\str{G}$ and there exist $(a,b)\in \mathcal F$, $w\in G$ such that $\{d(u,w),d(v,w)\}=\{a,b\}$. There are no other edges in $\str{G}'$.

Given $1\leq M\leq \delta$, a one-to-one function $t\colon \{1,2,\ldots,\delta\}\setminus \{M\}\to \mathbb N$ 
and a function $\mathbb F$
from $\{1,2,\ldots,\delta\}\setminus \{M\}$ to the power set of $\mathcal D$, we define the {\em $(\mathbb F,t,M)$-completion} of  $\str{G}$
as the limit of a sequence of edge-labelled graphs  $\str{G}_1, \str{G}_2,\ldots$ such that $\str{G}_1=\str{G}$ and $\str{G}_{k+1}=\str{G}_k$ if $t^{-1}(k)$ is undefined
and $\str{G}_{k+1}$ is the $(\mathbb F(t^{-1}(k)),t^{-1}(k))$-completion of $\str{G}_{k}$ otherwise, with every pair of vertices not forming an edge in this limit set to distance $M$.
\end{definition}

We will call the vertex $w$ from Definition \ref{defn:ftmcompletion} the {\em witness of the edge $u,v$}. The function $t$ is called the {\em time function} of the completion because edges of length $a$ are inserted to $\str{G}_{t(a)}$ the $t(a)$-th step of the completion. If for a $(\mathbb F, t, M)$-completion and distances $a,c$ there is a distance $b$ such that $(a,b)\in \mathbb F(c)$ (i.e. the algorithm might complete a fork $(a,b)$ with distance $c$), we say that {\em $c$ depends on $a$}.

\begin{definition}[Magic distances]
\label{defn:magicdistance}
Let $M\in\{1,2,\ldots, \delta\}$ be a distance. We say that $M$ is {\em magic} (with respect to $\Aclass$) if $$\max\left(K_1, \left\lceil\frac{\delta}{2}\right\rceil\right) \leq M \leq \min\left(K_2,\left\lfloor\frac{C-\delta-1}{2}\right\rfloor\right).$$
\end{definition}

Note that for primitive admissible parameters $(\delta, K_1,K_2,C_0,C_1)$ such an $M$ always exists.

\begin{observation}\label{obs:magicismagic}
The set $S$ of magic distances (with respect to $\Aclass$) is $$S=\left\{1\leq a \leq \delta : aab\text{ is allowed for all }1\leq b\leq \delta\right\}.$$
\end{observation}

\begin{proof}
If a distance $a$ is in $S$, then $a\geq K_1$ (otherwise the triangle $aa1$ has perimeter $2a+1$, which is odd and smaller than $2K_1+1$, hence forbidden by the $K_1$ bound), $a\geq \left\lceil\frac{\delta}{2}\right\rceil$ (otherwise the triangle $aa\delta$ is non-metric), $a\leq \left\lfloor\frac{C-\delta-1}{2}\right\rfloor$ (otherwise the triangle $aab$ has perimeter $C$ for $b = C -2a \leq \delta$), and $a\leq K_2$ (otherwise the triangle $aa1$ has odd perimeter and $2a\geq 2K_2+1$, hence is forbidden by the $K_2$ bound). The other inclusion follows from the definition of $\Aclass$.
\end{proof}

Let $M$ be a magic distance and $x\in\{1,\ldots,\delta\}\setminus\{M\}$. Define $$\mathcal F^+_x = \left\{(a,b)\in \mathcal D : a+b=x\right\},$$ $$\mathcal F^-_x = \left\{(a,b)\in \mathcal D : |a-b|=x\right\},$$ $$\mathcal F^C_x = \left\{(a,b)\in \mathcal D : C-1-a-b=x\right\}.$$ We further denote
$$\mathbb F_M(x) =
\begin{cases} 
      \mathcal F^+_x\cup \mathcal F^C_x & \text{if }x < M \\
      \mathcal F^-_x & \text{if }x > M.
\end{cases}
$$
For a magic distance $M$, we also define the function $t_M\colon  \{1,\ldots,\delta\}\setminus \{M\} \rightarrow \mathbb N$ as
$$t_M(x) =
\begin{cases} 
      2x-1 & \text{if } x < M \\
      2(\delta-x) & \text{if }x > M.
\end{cases}
$$
\begin{figure}
\centering
\includegraphics{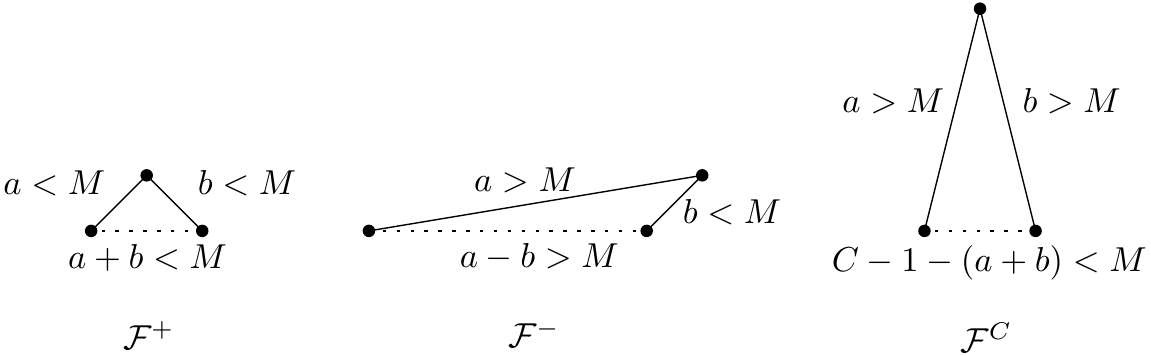}
\caption{Forks used by $\mathbb F_M$.}
\label{fig:Fforks}
\end{figure}
Forks and how they are completed according to $\mathbb F_M$ are schematically depicted in Figure~\ref{fig:Fforks}.

\begin{definition}[Completion with magic parameter $M$]
\label{defn:magiccompletion}
Let $M$ be a magic distance satisfying the following extra conditions:
\begin{enumerate}
\item If the parameters satisfy Case~\ref{III} with $K_1+2K_2 = 2\delta - 1$, then $M>K_1$;
\item if the parameters satisfy Case~\ref{III} and further $C'>C+1$ and $C=2\delta+K_2$, then $M<K_2$.
\end{enumerate}
We then call the $(\mathbb F_M,t_M,M)$-completion (of $\str{G}$) the {\em completion (of $\str{G}$) with magic parameter $M$}. 
\end{definition}

Our main goal of the following section is the proof of Theorem \ref{thm:magiccompletion} that shows that the completion of $\str{G}$ with magic parameter $M$ lies in $\Aclass$ if and only if $\str{G}$ has some completion in $\Aclass$.

The two extra conditions in Definition \ref{defn:magiccompletion} are a way to deal with certain extremal choices of admissible primitive parameters $(\delta, K_1,K_2,C_0,C_1)$. 

\begin{lemma}
For primitive parameters $(\delta, K_1,K_2,C_0,C_1)$ there is always an $M$ satisfying Definitions~\ref{defn:magicdistance} and \ref{defn:magiccompletion}.
\end{lemma}
\begin{proof}

For Case~\ref{III} with $K_1+2K_2 = 2\delta-1$, we proceed as follows: From admissibility, we have
\[\begin{split}K_1+2K_2&=2\delta-1\\ 3K_2&\geq2\delta\end{split}\] so we conclude that $K_1<K_2$. From this information and $K_1+2K_2=2\delta-1$, we derive $K_1<\frac{2}{3}\delta$. We know that $\delta-1 \geq \frac{2}{3}\delta$ for $\delta \geq 3$ and $K_1\leq \delta-2$, so $\left\lfloor\frac{C-\delta-1}{2}\right\rfloor \geq \left\lfloor\frac{\delta+K_1+1}{2}\right\rfloor\geq \frac{\delta+K_1}{2}$. Hence, $K_1 < \left\lfloor\frac{C-\delta-1}{2}\right\rfloor$ and there is always a magic number greater than $K_1$.\bigskip

In Case~\ref{III} with $C'>C+1$ and $C=2\delta+K_2$, we know from admissibility that $C>2\delta + K_1$, so $K_2 > K_1$. Now we need $\left\lceil\frac{\delta}{2}\right\rceil < K_2$. For $\delta \geq 3$, the inequality $\left\lceil\frac{\delta}{2}\right\rceil \leq \frac{2}{3}\delta$ holds with equality only for $\delta = 3$. Admissibility tells us $3K_2\geq 2\delta$. Now, if $\delta > 3$ or $K_2\neq \frac{2}{3}\delta$, it follows that $\left\lceil\frac{\delta}{2}\right\rceil < K_2$. 

The only remaining possibility is $\delta=3$ and $K_2=2$, which implies $C=8$ and $K_1=1$, which gives us $2K_2+K_1 = 5 = 2\delta-1$. The admissibility condition $C\geq 2\delta +K_1+2$ then yields $C\geq 9$, a contradiction. Hence there always is a magic number smaller than $K_2$.

If both these situations occur simultaneously, then we further require $M$ with $K_1<M<K_2$. But that follows as $C=2\delta+K_2$ and whenever $K_1+2K_2=2\delta-1$, from admissibility we have $C\geq 2\delta+K_1+2$, hence $K_2\geq K_1+2$.
\end{proof}

Observe that the algorithm only makes use of $C$, $\delta$ and $M$. The
interplay of individual parameters of algorithm is schematically depicted in
Figure~\ref{fig:algorithm}.
\begin{figure}
\centering
\includegraphics{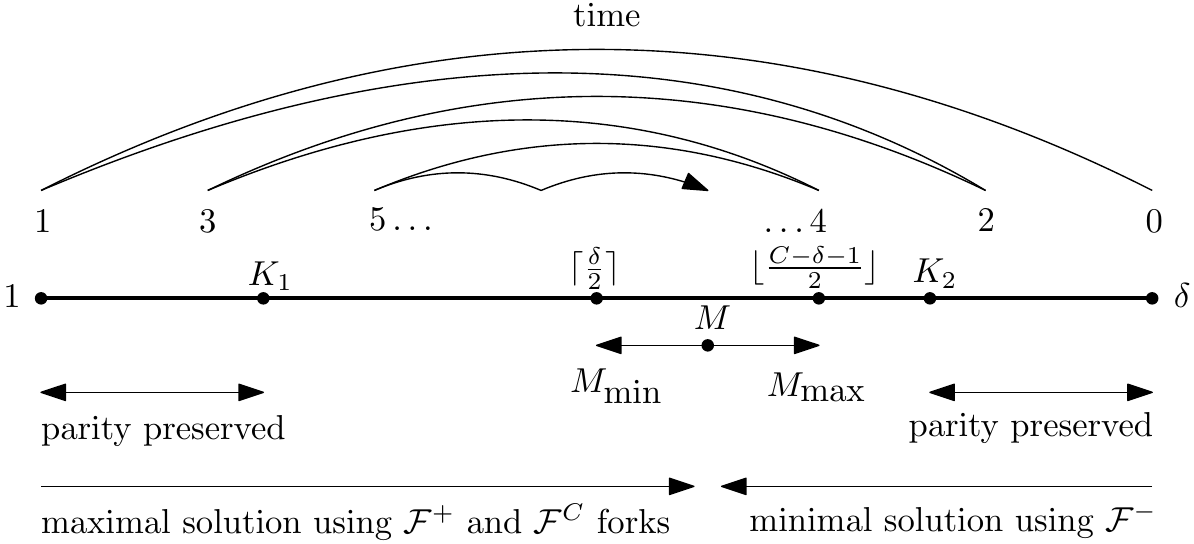}
\caption{A sketch of the main parameters of the completion algorithm, the Optimality Lemma~\ref{lem:bestcompletion} and the Parity Lemma~\ref{lem:sameparity}.}
\label{fig:algorithm}
\end{figure}

\begin{example}[Case~\ref{IIb}]
In our proofs, Case~\ref{IIb} will often form a special case.
The smallest (in terms of diameter) set of acceptable parameters that is in Case \ref{IIb} is: $$\delta=5, C=C_1 = 13, C^\prime=C_0 = 16, K_1 = K_2 = \frac{2\delta-1}{3} = 3.$$ Here $M=3$, and it is the only choice for a magic number.

Forbidden triangles are those that are non-metric (113, 114, 115, 124, 125, 135, 225), or forbidden by the $K_1$-bound $(111,122)$, the $K_2$-bound (144, 155, 245), or the $C_1$-bound (355, 445, 555). There are no triangles forbidden by the $C_0$-bound.
\begin{figure}
\centering
\includegraphics{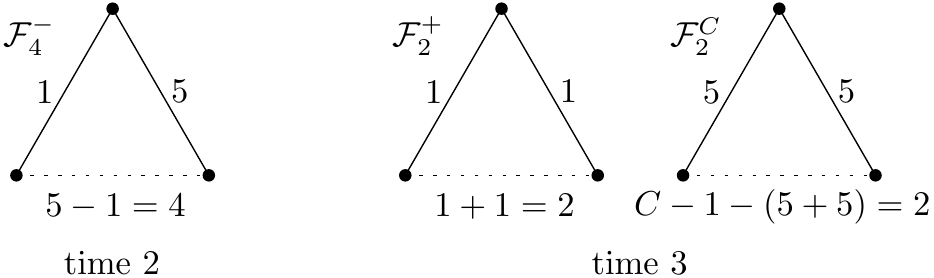}
\caption{Forks considered by the algorithm to complete to $\mathcal A^5_{3,3,16,13}$ with $M=3$.}
\label{fig:forks}
\end{figure}
\begin{table}[t]
\centering
\begin{tabular}{|c|c|c|c|c|c|}
\hline
&$j=1$&$j=2$&$j=3$&$j=4$&$j=5$\\ \hline
$i=1$&\textbf{2} &  $1,\textbf{3}$  &  $2,\textbf{3},4$    & $\textbf{3},5$   		& \textbf{4} 		\\ 
$i=2$&					&  $2,\textbf{3},4$&  $1,2,\textbf{3},4,5$& $2,\textbf{3},4$ 		& $\textbf{3},5$ 	\\ 
$i=3$&					&  							&  $1,2,\textbf{3},4,5$& $1,2,\textbf{3},4,5$	& $2,\textbf{3},4$	\\ 
$i=4$&					&  							&  									&  $2,\textbf{3},4$		& $1,\textbf{3},5$	\\ 
$i=5$&					&  							&  									&  									& $\textbf{2},4$		\\ \hline
\end{tabular}
\caption{Possible ways to complete $(i,j)$ forks, the bold number is the completion with magic parameter $M = 3$.}
\label{tab:forks}
\end{table}
Table~\ref{tab:forks} lists all possible completions of forks, with the completion preferred by our algorithm in bold type. Completions of forks in this class are depicted in Figure~\ref{fig:forks}.
Notably, the magic distance $M=3$ is chosen for all forks except $(1,1)$, which is completed by $d^+((1,1)) = 2$, $(1,5)$, which is completed by $d^-((1,5)) = 4$, and $(5,5)$, which is a $C$-bound case. 
Those cases are the only forks where $M=3$ cannot be chosen, so instead the algorithm chooses the nearest possible completion. What makes Case~\ref{IIb} special is the situation where one can choose $M-1$ or $M+1$ but not $M$ when completing a fork (for $\delta=5$ it is the fork $(5,5)$, as both the triangles $5,5,2$ and $5,5,4$ are allowed, while $5,5,3$ is forbidden by the $C_1$ bound; this behavious is going to force us to deal with some corner cases later). 

The algorithm will thus effectively run in three steps.  First (at time 2) it will complete all forks $(1,5)$ with distance 4, next (at time 3) it will complete all forks $(1,1)$ and $(5,5)$ with distance 2 and finally it will turn all non-edges into edges of distance 3. Examples of runs of this algorithm are given later, see Figures~\ref{fig:1555} and~\ref{fig:11555}.
\end{example}

\section{What do forbidden triangles look like?}
\label{sec:forbtriangles}
The majority of the proofs in the following sections assume that the completion algorithm with magic parameter $M$ introduces some forbidden triangle and then we argue that the triangle must be forbidden in any completion, hence the input structure has no completion into $\Aclass$. In such an argument it will be helpful to know how the different types of forbidden triangles relate to the magic parameter $M$. We will use $a,b,c$ for the lengths of the edges of the triangle and without loss of generality assume $a\leq b\leq c$. All conclusions are summarised in Figure~\ref{fig:Ftriangles}.
\begin{figure}[t]
\centering
\includegraphics{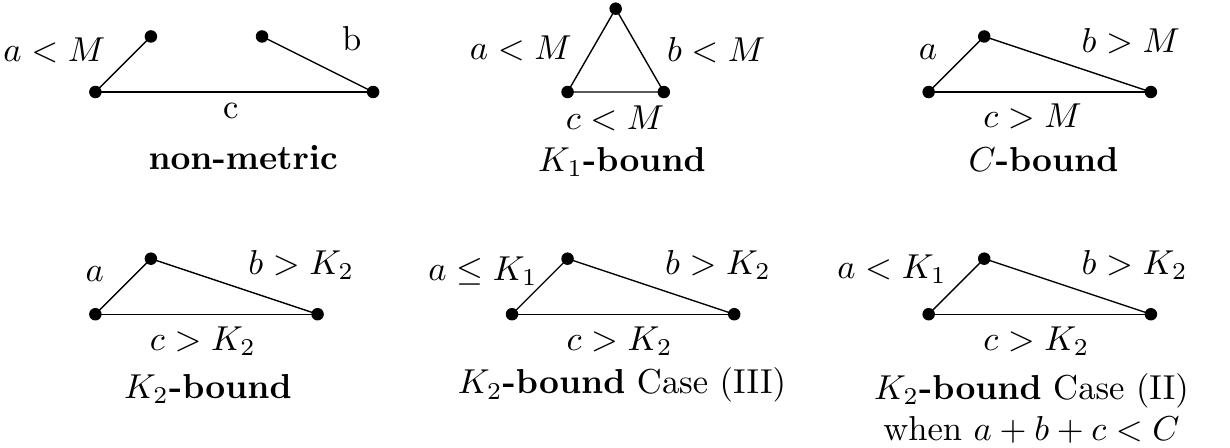}
\caption{Types of forbidden triangles.}
\label{fig:Ftriangles}
\end{figure}

\begin{description}
\item[non-metric:]
If $a+b<c$, then $a < M$, because otherwise $a+b\geq 2M \geq \delta$.

\item[$K_1$-bound:]
If $a+b+c < 2K_1+1$, $a+b+c$ is odd and $abc$ is metric, then $a,b,c < K_1\leq M$, because if $c\geq K_1$, then from the metric condition $a+b\geq c\geq K_1$ and hence $a+b+c\geq 2K_1$, for odd $a+b+c$ this means $a+b+c\geq 2K_1+1$.

\item[$C$-bound:]
If $a+b+c\geq C$ then $b,c > M$. Suppose for a contradiction that $a,b\leq M$. We then have $a+b\geq C-c\geq C-\delta$, but on the other hand $a+b\leq 2M\leq 2\left\lfloor \frac{C-\delta-1}{2} \right\rfloor\leq C-\delta-1$, which together yield $C-\delta-1\geq C-\delta$, a contradiction. Note that in some cases $C'\neq C+1$, but this observation still holds, as it only uses $a+b+c\geq C$.

\item[$K_2$-bound:]
If $abc$ is a metric triangle with odd perimeter, then $abc$ breaks the $K_2$ condition if and only if $b+c\geq 2K_2+a+1$ (the 1 on the right side comes from $a+b+c$ being odd and all distances being integers). Then $b,c>K_2$, because if $b\leq K_2$, from metricity we have $c\leq a+b$, hence $a+2K_2\geq (a+b)+b \geq c+b\geq 2K_2+a+1$, a contradiction.

Moreover, in Case~\ref{III} of Theorem~\ref{thm:admissible} we have $a \leq K_1$ because if $a > K_1$, we have $b+c\geq 2K_2+a+1 > 2K_2+K_1+1$ and from admissibility conditions for Case~\ref{III} we have $2K_2+K_1\geq 2\delta-1$, which gives $b+c>2\delta$, a contradiction. (Note that if $2K_2+K_1>2\delta-1$, we have $a<K_1$.)

Finally if $a+b+c < C$ (which is stronger than not being forbidden by the $C$ bound, as it also includes the $C'>C+1$ cases) and we are in Case~\ref{II} (where $C=2K_1+2K_2+1$), we get $a < K_1$, because if $a\geq K_1$, we would get $a+b+c\geq 2K_2+2a+1\geq 2K_2+2K_1+1 = C$.

Note that later we shall refer to all the corner cases mentioned in these paragraphs.
\end{description}

\section{Basic properties of the algorithm}
Recall the definition of $t_M$ and $\mathbb F_M$:
$$t_M(x) =
\begin{cases} 
      2x-1 & \text{if } x < M \\
      2(\delta-x) & \text{if }x > M.
\end{cases}
$$
$$\mathbb F_M(x) =
\begin{cases} 
      \mathcal F^+_x\cup \mathcal F^C_x & \text{if }x < M \\
      \mathcal F^-_x & \text{if }x > M.
\end{cases}
$$
Intuitively, the function $\mathbb F_M$ selects the forks that will be completed to triangles with an edge of type $t^{-1}_M(x)$ at time $x$. At time 0 it looks for forks that can be completed with distance $\delta$, then with distance 1, jumping back and forth on the distance set and approaching $M$ (cf. Figure \ref{fig:algorithm}). Observe that all forks that cannot be completed with $M$ are in some $\mathbb F_M(x)$.

Now we shall precisely state and prove that $t_M$ gives a suitable injection for the algorithm, as claimed before Definition~\ref{defn:ftmcompletion}.

\begin{lemma}[Time Consistency Lemma]\label{lem:expandtime}
Let $a,b$ be distances different from $M$. If $a$ depends on $b$, then $t_M(a) > t_M(b)$.
\end{lemma}

\begin{proof}
We consider three types of forks used by the algorithm:
\begin{description}
\item[$\mathcal F^+$:]
If $a<M$ and $\mathcal F^+_a\neq\emptyset$, then $b<a<M$, hence $t_M(b) < t_M(a)$.

\item[$\mathcal F^C$:]
If $a < M$ and $\mathcal F^C_a\neq\emptyset$, then we must have $b,c > M$ ($(b,c)\in \mathcal F^C_a$ with $a=C-1-b-c$). Otherwise, if for instance $b\leq M$, then $C-\delta-1\leq C-1-c = a+b < 2M \leq 2\left\lfloor \frac{C-\delta-1}{2} \right\rfloor$, a contradiction. As $C\geq 2\delta+2$ (we are dealing with the primitive case), we obtain the inequality $b = (C-1)-c-a \geq (2\delta + 1) - \delta - a = \delta+1-a$. Hence $t_M(b) \leq 2(a-1) < 2a-1 = t_M(a)$.

\item[$\mathcal F^-$:]
Finally, we consider the case where $a > M$ and $\mathcal F^-_a\neq\emptyset$. Then either $a = b-c$, which implies $b>a>M$ and thus $t_M(b)<t_M(a)$, or $a = c-b$, which means $b = c-a\leq \delta-a$. Because of $a>M\geq \left\lceil\frac{\delta}{2}\right\rceil$, we have $b<M$. So $t_M(b) \leq 2(\delta-a) - 1 < 2(\delta-a) = t_M(a)$.
\end{description}
\end{proof}

\begin{lemma}[$\mathbb F_M$ Completeness Lemma]\label{lem:misgood}
Let $\str{G}\in \mathcal G^\delta$ and $\overbar{\str{G}}$ be its completion with magic parameter $M$. If there is a forbidden triangle (w.r.t. $\Aclass$) or a triangle with perimeter at least $C$ in $\overbar{\str{G}}$ with an edge of length $M$, then this edge is also present in $\str{G}$.
\end{lemma}

Observe that for $C'\neq C+1$, this lemma is talking not only about forbidden triangles, but about all triangles with perimeter at least $C$.

\begin{proof}
By Observation \ref{obs:magicismagic} no triangle of type $aMM$ is forbidden, so suppose that there is a forbidden triangle $abM$ in $\overbar{\str{G}}$ such that the edge of length $M$ is not in $\str{G}$. For convenience define $t_M(M) = \infty$, which corresponds to the fact that edges of length $M$ are added in the last step.

\begin{description}
\item[non-metric:] If $abM$ is non-metric then either $a+b<M$ or $|a-b|>M$. By Lemma \ref{lem:expandtime} we have in both cases that $t_M(a+b)$ (respectively, $t_M(|a-b|)$) is greater than both $t_M(a)$ and $t_M(b)$. Therefore the completion algorithm would chose $a+b$ (resp, $|a-b|$) as the length of the edge instead of $M$.

\item[$K_1$-bound:] Now that we know that $abM$ is metric, we also know that it is not forbidden by the $K_1$ bound, because $M\geq K_1$.

\item[$C$-bound:] If $a+b+M\geq C$ (which includes all the triangles forbidden by $C_0$ or $C_1$ bounds), then $t_M(C-1-a-b)>t_M(a),t_M(b)$ by Lemma \ref{lem:expandtime}, so the algorithm would set $C-1-a-b$ instead of $M$ as the length of the third edge.

\item[$K_2$-bound:] Finally we deal with the $K_2$ bound. Suppose that $abM$ is metric, its perimeter is less than $C$, and it is forbidden by the $K_2$ bound. From Section \ref{sec:forbtriangles} we have that the two long edges have to be longer that $K_2$, and the shortest edge is at most $K_1$ with equality only in Case~\ref{III} with $K_1+2K_2=2\delta-1$.

As $M\leq K_2,$ we know that $M$ is the shortest edge. But also $M\geq K_1$, hence this situation can happen only when $K_1$ is the length of the shortest edge, which is only in Case~\ref{III} with $K_1+2K_2=2\delta-1$. But from Definition~\ref{defn:magiccompletion} we have in this case $M>K_1$. Hence this situation never occurs.
\end{description}
\end{proof}

It may seem strange that the algorithm does not differentiate between $C_0$ and $C_1$. The following observation justifies this
by showing that in the case where $C'>C+1$, these bounds have a relatively limited effect on the run of the algorithm.
\begin{observation}\label{obs:FCforks}
If $C'>C+1$, then either $\mathcal F^C_x$ is empty for all $x < M$ or the parameters satisfy $\ref{IIb}$. In the latter case, only $\mathcal F^C_{M-1} = \{ (\delta,\delta) \}$ is non-empty. Furthermore, in this case $t_M(M-1)$ is the maximum of the time-function. This implies that $(\delta,\delta)$-forks are completed to $M-1$ in the penultimate step of the completion algorithm.
\end{observation}
\begin{proof}
Consider a fork $(a,b) \in \mathcal F^C_x$ and the cases where $C'>C+1$ is allowed.

In Case~\ref{III} with $C'>C+1$ we have (by admissibility) $C\geq 2\delta+K_2$, so $x = C-1-a-b\geq K_2-1$ with equality only for $C=2\delta+K_2$. From the extra condition for a magic parameter in Definition~\ref{defn:magiccompletion}, we get that $M < K_2$.

In Case~\ref{IIb} we have $M=K_2=K_1=\frac{2\delta-1}{3}$, hence $C=2K_1+2K_2+1 = 2\delta+K_2$, thus again we have $C-1-a-b\geq K_2-1$. This means that the only fork in $\mathcal F^C_{M-1}$ is going to be $(\delta,\delta)$, which will be completed by $K_2-1=M-1$.

In order to see that $t_M(M-1)$ is maximal, it is enough to check $t_M(M+1) < t_M(M-1)$. We have $3M=3K_2=2\delta-1$, so by definition $t_M(M-1) = 2M-3$ and $t_M(M+1) = 2\delta-2M-2$, so we want $2M-3>2\delta-2M-2$, or $4M>2\delta+1$ which is true for $\delta\geq 5$ and this always holds in Case \ref{IIb}. So $t_M(M-1) > t_M(M+1)$ and therefore $t_M(M-1)>t_M(a)$ for any $a$ different from $M$ and $M-1$. 
\end{proof}

\begin{lemma}[Optimality Lemma]\label{lem:bestcompletion}
Let $\str{G}=(G,d)\in \mathcal G^\delta$ such that it has a completion in $\Aclass$. Denote by $\overbar{\str{G}}=(G,\bar d)$ the completion of $\str{G}$ with magic parameter $M$ and let $\str{G}'=(G,d')\in\Aclass$ be an arbitrary completion of $\str{G}$. 
Then for every pair of vertices $u,v\in G$ one of the following holds:
\begin{enumerate}
 \item $d'(u,v) \geq \bar d(u,v) \geq M$,
 \item $d'(u,v) \leq \bar d(u,v) \leq M$,
 \item the parameters $(\delta,K_1,K_2,C_0,C_1)$ satisfy Case~\ref{IIb}, $\bar d(u,v) = M-1$, $d'(u,v) > M$ and $d'(u,v)$ has the same parity as $\bar d(u,v)$.
\end{enumerate}
\end{lemma}
Note that for $\bar d(u,v)=M$ the statement trivially holds.

\begin{proof}
Suppose that the statement is not true and take any witness $\str{G}'=(G, d')$ (i.e. a completion of ${\str{G}}$ into $\Aclass$ such that there is a pair of vertices violating the statement). Recall that the completion with magic parameter $M$ is defined as a limit of a sequence $\str{G}_1, \str{G}_2, \ldots$ of edge-labelled graphs such that $\str{G}_1={\str{G}}$ and each two subsequent graphs differ at most by adding edges of a single distance.

Take the smallest $i$ such that in the graph $\str{G}_i = (G,d_i)$ there are vertices $u,v\in G$ with $d_i(u,v) > M$ and $d_i(u,v) > d'(u,v)$ or $d_i(u,v) < M$ and $d_i(u,v) < d'(u,v)$. Let $w\in G$ be the witness of $d_i(u,v)$. In Case~\ref{IIb}, by Observation~\ref{obs:FCforks} edges of length $M-1$ are added in the last step of our completion algorithm. Therefore we know that the distances $d_{i-1}(u,w)$ and $d_{i-1}(v,w)$ satisfy the optimality conditions in point 1 or 2.

We shall distinguish three cases, based on whether $d_{i}(u,v)$ was introduced by $\mathcal F^-$, $\mathcal F^+$ or $\mathcal F^C$:

\paragraph{$\mathcal F^-$ case} We have $M < d_i(u,v) = |d_{i-1}(u,w)-d_{i-1}(v,w)|$. Without loss of generality let us assume $d_{i-1}(u,w) > d_{i-1}(v,w)$, which means that $d_{i-1}(u,w) > M$ and $d_{i-1}(v,w) < M$ (as $M\geq \left\lceil\frac{\delta}{2}\right\rceil$). From the minimality of $i$, it follows that $d'(u,w) \geq d_{i-1}(u,w)$ and $d'(v,w)\leq d_{i-1}(v,w)$. Since $\str{G}'$ is metric we have $d_i(u,v) = d_{i-1}(u,w)-d_{i-1}(v,w) \leq d'(u,w)-d'(v,w) \leq d'(u,v)$, which is a contradiction.

\paragraph{$\mathcal F^+$ case} We have $M > d_i(u,v) = d_{i-1}(u,w)+d_{i-1}(v,w)$, hence $d_{i-1}(u,w),\allowbreak d_{i-1}(v,w)<M$. By the minimality of $i$ we have $d'(u,w)\leq d_{i-1}(u,w)$ and $d'(v,w)\leq d_{i-1}(v,w)$. Since $\str{G}'$ is metric, we get $d'(u,v)\leq d_i(u,v)$, which contradicts our assumptions.

\paragraph{$\mathcal F^C$ case} We have $M > d_i(u,v) = C-1-d_{i-1}(u,w)-d_{i-1}(v,w)$.

First suppose that $C'=C+1$. Recall that, by the admissibility of $C$, we have $C-1\geq 2\delta+1$ and $M\leq \left\lfloor \frac{C-\delta-1}{2} \right\rfloor$. Thus we get $d_{i-1}(u,w),d_{i-1}(v,w)>M$ (otherwise, if, say, $d_{i-1}(u,w)\leq M$, we obtain the contradiction $C-\delta-1\geq 2M > d_{i-1}(u,w)+d_i(u,v) = C-1-d_{i-1}(v,w) \geq C-\delta-1$). So again $d'(u,w)\geq d_{i-1}(u,w)$ and $d'(v,w)\geq d_{i-1}(v,w)$, which means that the triangle $u,v,w$ in $\str{G}'$ is forbidden by the $C$ bound, which is absurd as $\str{G}'$ is a completion of $\str{G}$ in $\Aclass$.

\medskip

It remains to discuss the case where $C' > C+1$. By Observation~\ref{obs:FCforks}, we only need to consider Case~\ref{IIb}, $d_{i}(u,v) = K_2-1 = M-1$ and $d_{i-1}(u,w) = d_{i-1}(v,w) = \delta$. By our assumption we have $d'(u,v) > d_i(u,v)$. Hence if $d'(u,v) \geq M$ it has to have the same parity as $d_i(u,v)$ (otherwise the triangle $u,v,w$ would be forbidden in $\str{G'}$ by the $C$ bound).
\end{proof}

Next we show that the algorithm initially runs in a way that preserves the parity of completions to $\Aclass$.

\begin{lemma}[Parity Lemma]\label{lem:sameparity}
Let $\str{G}$, $\overbar{\str G}$ and $\str{G'}$ be as in Lemma \ref{lem:bestcompletion}. Then for every pair of vertices $u,v\in G$ such that either $\bar d(u,v) \leq \min(K_1,M-1)$ or $\bar d(u,v) \geq \max(K_2,M+1)$, at least one of the following holds:
\begin{enumerate}
\item The parity of $\bar d(u,v)$ is the same as the parity of $d'(u,v)$;
\item the parameters come from Case~\ref{III}, $C=2\delta+K_1+1$, $C\neq  2K_1+2K_2+1$, $M>K_1>1$ and $\bar d(u,v)=K_1$.
\end{enumerate}
\end{lemma}

Note that we are only interested in distances not equal to $M$. 

\begin{proof}
Suppose that the statement is not true, and let $\str{G}'=(G, d')$ be a counterexample. Recall that the completion with magic parameter $M$ is defined as a limit of a sequence $\str{G}_1, \str{G}_2, \ldots$ of edge-labelled graphs  such that $\str{G}_1={\str{G}}$ and each two subsequent edge-labelled graphs differ at most by adding edges of a single distance.

Take the smallest $i$ such that in $\str{G}_i = (G,d_i)$ there are vertices $u,v\in G$ with $d_i(u,v)$ and $d'(u,v)$ not satisfying the lemma. Denote by $w$ a witness of the distance $d_i(u,v)$. As in the proof Lemma~\ref{lem:bestcompletion}, we can argue that $d_{i-1}(u,w)$ respectively $d_{i-1}(v,w)$ satisfy the optimality conditions 1 or 2 in Lemma~\ref{lem:bestcompletion}.

First we will show that the exceptional case 2 from the statement only happens at the very end of the induction, hence when using the induction hypothesis (or minimality of $i$), we can work only with the first part of the statement.

Suppose that the parameters satisfy Case~\ref{III} and further $C=2\delta+K_1+1$, $C\neq 2K_1+2K_2+1$ and $M>K_1>1$. We have $t_M(K_1) > t_M(a)$ for any distance $a<K_1$ and also, by admissibility, $t_M(K_1) > t_M(b)$ for any distance $b\geq K_2$ and $b>M$: since $t_M(K_1) = 2K_1-1$ and $t_M(b) \leq 2\delta-2K_2$, we need to verify that $2K_1-1 > 2\delta-2K_2$ and thus $2K_1+2K_2 > 2\delta+1$. By admissibility it follows $2K_2+K_1\geq 2\delta$ (when $2K_2+K_1=2\delta-1$, admissibility implies $C\geq 2\delta+K_1+2$), which give the desired bound.

\medskip

Next observe that if $K_1=1$ then from Lemma \ref{lem:bestcompletion} we have that whenever $\bar d(u,v)=1$ for some vertices $u,v$, then in any completion the edge has also length 1, hence also fixed parity.

As in the proof of Lemma \ref{lem:bestcompletion}, we will now distinguish three cases based on whether $d_i(u,v)$ was introduced due to $\mathcal F^+$, $\mathcal F^-$ or $\mathcal F^C$:

\paragraph{$\mathcal F^+$ case} In this case $d_i(u,v)<M$ and $d_i(u,v)=d_{i-1}(u,w)+d_{i-1}(w,v)$. Because of our assumption $d_i(u,v)\leq K_1$, the perimeter of the triangle $uvw$ in $\str{G}_i$ is even and at most $2K_1$. By Lemma~\ref{lem:bestcompletion} either the third possibility happened, hence $\bar{d}(u,v)$ has the same parity as $d'(u,v)$, or we have $d'(u,v)\leq d_{i}(u,v)$, $d'(u,w)\leq d_{i-1}(u,w)$ and $d'(v,w)\leq d_{i-1}(v,w)$, hence $d'(u,v)+d'(u,w)+d'(w,v)$ is odd and smaller than $2K_1+1$. Thus the triangle $uvw$ is forbidden by the $K_1$ bound in $\str{G}'$, a contradiction.

\paragraph{$\mathcal F^-$ case} Here $d_i(u,v) > M$ and without loss of generality we can assume that $d_i(u,v)=d_{i-1}(u,w)-d_{i-1}(w,v)$. Then the triangle $uvw$ has even perimeter with respect to $d_i$. By our assumption we have $d_i(u,v)\geq K_2$ and thus $d_{i-1}(u,w)>d_i(u,v)\geq K_2$ and $d_{i-1}(v,w)<M$.

This implies $d_i(u,v)+d_{i-1}(u,w) = 2d_i(u,v)+d_{i-1}(v,w)\geq 2K_2+d_{i-1}(v,w)$. From Lemma \ref{lem:bestcompletion} we get that $d'(u,v)\geq d_i(u,v)$, $d'(u,w)\geq d_{i-1}(u,w)$ and $d'(v,w)\leq d_{i-1}(v,w)$, hence also $d'(u,v)+d'(u,w) \geq 2K_2+d'(v,w)$ holds. Thus the triangle $uvw$ is forbidden by the $K_2$ bound in $\str{G}'$, a contradiction.

\paragraph{$\mathcal F^C$ case} Here $d_i(u,v)<M$ and $d_i(u,v)=C-1-d_{i-1}(u,w)-d_{i-1}(w,v)$. From our assumption it follows that $d_i(u,v)\leq K_1$.

In Case~\ref{III} we have $C\geq 2\delta+K_1+1$, hence $d_i(u,v) = K_1$ if and only if $C=2\delta+K_1+1$, $M>K_1$ and $d(u,w)= d(v,w) = \delta$; this case is treated in point 2.

It remains to consider Case~\ref{II}. Hence we can assume that $d'(u,v) < d(u,v)$ and these edges have different parity. Note that then the triangle $uvw$ has even perimeter $C-1$. By Lemma~\ref{lem:bestcompletion} we have $d'(u,w)+d'(v,w)\geq d_{i-1}(u,w)+d_{i-1}(v,w) = C-1-d(u,v) = 2K_1+2K_2-d_i(u,v)$. But as $d'(u,v) \leq d_i(u,v)\leq K_1$ we have $d'(u,w)+d'(v,w)\geq 2K_2+d'(u,v)$, so the triangle $uvw$ is forbidden by the $K_2$ bound in $\str{G}'$, a contradiction.
\end{proof}

\begin{lemma}[Automorphism Preservation Lemma]
\label{lem:aut}
Let $\str{G}\in \mathcal G^\delta$ and let $\overbar{\str{G}}$ be its completion with magic parameter $M$. Then every automorphism of $\str{G}$ is also an automorphism of $\overbar{\str{G}}$.
\end{lemma}
\begin{proof}
Given $\str{G}$ and an automorphism $f\colon \str{G}\to \str{G}$, it can be verified by induction that for every $k>0$, $f$ is also an automorphism graph $\str{G}_k$ as in Definition~\ref{defn:ftmcompletion}.
For every edge $x,y$ of $\str{G}_k$ which is not an edge of $\str{G}_{k-1}$, it is true that $f(x),f(y)$
is also an edge of $\str{G}_k$ which is not an edge of $\str{G}_{k-1}$, and moreover the edges $x,y$ and $f(x),f(y)$ are of the same length. This follows
directly from the definition of $\str{G}_k$.
\end{proof}

\section{Correctness of the completion algorithm}
\label{sec:magiccompletion}
In the next five lemmas we use Lemmas \ref{lem:bestcompletion} and \ref{lem:sameparity} to show that $\str{G}\in \mathcal G^\delta$ has a completion into $\Aclass$, if and only if the algorithm with a magic parameter $M$ yields such a completion. We deal with each type of forbidden triangle separately, and in doing that, we implicitly use the results of Section~\ref{sec:forbtriangles}.

\begin{lemma}[$C$-bound Lemma]\label{lem:Cbound}
Suppose $C'=C+1$, and let ${\str{G}}=(G,{d})\in \mathcal G^\delta$ be such that there is a completion of ${\str{G}}$ into $\Aclass$; let $\overbar{\str{G}}=(G,\bar d)$ be its completion with magic parameter $M$. Then there is no triangle forbidden by the $C$ bound in $\overbar{\str{G}}$.
\end{lemma}

\begin{proof}
Suppose for a contradiction that there is a triangle with vertices $u,v,w$ in $\overbar{\str{G}}$ such that $\bar d(u,v)+\bar d(v,w)+\bar d(u,w)\geq C$. For brevity let $a=\bar d(u,v)$, $b=\bar d(v,w)$ and $c=\bar d(u,w)$. Assume without loss of generality that $a\leq b\leq c$. Let $a',b',c'$ be the corresponding edge lengths in an arbitrary completion of ${\str{G}}$ into $\Aclass$. Then two cases can appear. 

Either $a,b,c > M$, and then by Lemma \ref{lem:bestcompletion} we have $a'\geq a$, $b'\geq b$ and $c'\geq c$, so we get the contradiction $a' + b' + c' \geq C$; or $a\leq M$, $c\geq b>M$ and $a+b+c\geq C$. In this case Lemma \ref{lem:bestcompletion} implies $b'\geq b$ and $c'\geq c$ and $a'\leq a$. If the edge $(u,v)$ was already in $\str{G}$, then clearly $a' + b' + c' \geq a+b+c\geq C$, which is a contradiction. If $(u,v)$ was not already an edge in $\str{G}$, then it was added by the completion algorithm with magic parameter $M$ in step $t_M(a)$. Let $\bar{a}=C-1-b-c$. Then clearly $\bar{a}<a$, which means that $t_M(\bar{a}) < t_M(a)$, and as $\bar{a}$ depends on $b,c$, we have $t_M(b),t_M(c)<t_M(\bar{a})$. But then the completion with magic parameter $M$ actually sets the length of the edge $u,v$ to be $\bar{a}$ in step $t_M(\bar{a})$, which is a contradiction.
\end{proof}

\begin{lemma}[Metric Lemma]\label{lem:metric}
Let $\str{G}$ and $\overbar{\str{G}}$ be as in Lemma \ref{lem:Cbound}. Then there are no non-metric triangles in $\str{G}$.
\end{lemma}

\begin{proof}
Suppose for a contradiction that there is a triangle with vertices $u,v,w$ in $\str{G}$ such that $d(u,v)+d(v,w)<d(u,w)$. Denote $a=d(u,v)$, $b=d(v,w)$ and $c=d(u,w)$ and assume without loss of generality that $a\leq b < c$. Let $a',b',c'$ be the corresponding edge lengths in an arbitrary completion of $\bar{\str{G}}$ into $\Aclass$. We shall distinguish three cases based on Section~\ref{sec:forbtriangles}:

\begin{enumerate}

\item First suppose $a,b,c < M$. Then $t_M(a)\leq t_M(b) < t_M(a+b) < t_M(c)$, which means that $c$ must be already in $\str{G}$. Note that in Case~\ref{IIb} if $b=K_1-1=M-1$, then $c\geq M$, hence we can use Lemma~\ref{lem:bestcompletion} for $a$ and $b$, which gives us that $a' + b' \leq a + b < c = c'$, which is a contradiction.

\item Another possibility is $a<M$ and $b,c\geq M$ (actually $c>M$, since $abc$ is non-metric).

Suppose $a'\leq a$ and $c'\geq c$ (the first possibility of Lemma~\ref{lem:bestcompletion}). If $b$ was already in $\str{G}$, then $\str{G}$ has no completion -- a contradiction. Otherwise clearly $c-a > b \geq M$, so $t_M(c-a) < t_M(b)$. But as $c-a$ depends on $c$ and $a$, we get $t_M(c-a) > t_M(c),t_M(a)$, which means that the completion algorithm with magic parameter $M$ would complete the edge $v,w$ with the length $c-a$ and not with $b$.

If the previous paragraph does not apply we have Case~\ref{IIb} and $a=K_1-1=K_2-1$. But then as $M=K_2$, we have $b\geq K_2$, which means $a+b\geq 2K_2-1 = \frac{4\delta-2}{3}-1\geq \delta$ for $\delta\geq 5$, which holds in \ref{IIb}, but that means that $abc$ is actually metric, a contradiction.

\item The last possibility is $a,b<M$ and $c\geq M$. Then either (by Lemma~\ref{lem:bestcompletion} and Lemma~\ref{lem:misgood} if $c=M$) we have $a'\leq a$, $b'\leq b$ and $c'\geq c$, hence the triangle $a',b',c'$ is again non-metric, or we have Case~\ref{IIb}, $b=K_1-1$, $a\leq K_1-1$.   The rest of proof of this lemma consists of a verification of this special case.
\end{enumerate}

From admissibility of \ref{IIb} we have $M=K_1=K_2=\frac{2\delta-1}{3}$ and $\delta\geq 5$. Note that $c-a \geq b+1 = K_1 = M$ from non-metricity of $abc$, hence $c>M$.

If both $a$ and $b$ were already in $\bar{\str{G}}$, then $abc$ is non-metric in any completion by Lemma~\ref{lem:bestcompletion}. The same thing is true if $b$ was already in $\bar{\str{G}}$ and $a'\leq a$ in any completion (i.e. either $a<K_1-1$ or $a$ was not introduced by $\mathcal F^C$ due to a $(\delta,\delta)$ fork).

Note that for $\delta\geq 8$ it cannot happen that $a=b=K_1-1$, as then $a+b=2K_1-2=2\frac{2\delta-1}{3} - 2 \geq \delta$, hence $a+b<c$ is absurd. So the only case when $a=b=K_1-1$ is $\delta=5$ (because from \ref{IIb} it follows that $\delta=3m+2$ for some $m\geq 1$). In that case we have triangle $5,2,2$ and each of the twos either was in $\bar{\str{G}}$ or is supported by a fork $(1,1)$ or by a fork $(5,5)$. And it can be shown that none of these structures has a strong completion into $\Aclass$.

Hence $b$ was not in the input graph and $a<K_1-1$.

Observe that $c-a = M$. From non-metricity of $abc$ we have $c-a \geq b+1 = K_1=M$. And if $c-a\geq M+1$, then $t_M(c-a)\leq t_M(M+1) = 2\delta-2M-2$. And this is strictly less than $t_M(M-1) = 2M-3$ since $M=\frac{2\delta-1}{3}$ and $\delta\geq 5$. Further as $M=\frac{2\delta-1}{3}$ is odd, we see that $a$ and $c$ have different parities.

From Lemma~\ref{lem:bestcompletion} we have that in any completion $c'\geq c$ and $a'\leq a$. So the only way that the triangle $u,v,w$ can be metric is to have $b'> b$. Note that $c'-a'\geq c-a = M = K_2$, hence $c'\geq a'+K_2$. And from Lemma~\ref{lem:sameparity} we have that the parities of $a,b,c$ are preserved.

Note that as $M$ is odd, $b'$ is even. And since the parities of $c'$ and $a'$ are different, we have that $a'+b'+c'$ is odd. Also note that $c'+b'\geq a'+K_2+K_2+1 \geq 2K_2+a'$. Hence $u,v,w$ is forbidden by the $K_2$ bound in $\str{G}'$, which is a contradiction.
\end{proof}

\begin{lemma}[$K_1$-bound Lemma]
\label{lem:K1bound}
Let $\str{G}$, $\overbar{\str{G}}$ be as in Lemma \ref{lem:Cbound}. Then there are no triangles forbidden by the $K_1$-bound in $\overbar{\str{G}}$.
\end{lemma}

\begin{proof}
Suppose for a contradiction that there is a metric (from Lemma~\ref{lem:metric} we already know that all triangles in $\str{G}$ are metric) triangle with vertices $u,v,w$ in $\str{G}$ such that $\bar d(u,v)+\bar d(v,w)+\bar d(u,w)$ is odd and less than $2K_1+1$. Denote $a=\bar d(u,v)$, $b=\bar d(v,w)$ and $c=\bar d(u,w)$. From Section~\ref{sec:forbtriangles} we get $a,b,c < K_1\leq M$.

First suppose that Lemma~\ref{lem:bestcompletion} gives us that for any completion $a',b',c'$ that $a'\leq a$, $b'\leq b$ and $c'\leq c$. Also $a$ has the same parity as $a'$, $b$ as $b'$ and $c$ as $c'$ by Lemma~\ref{lem:sameparity}, hence $a'+b'+c'\leq a+b+c$ and those two expressions have the same parity, hence $a',b',c'$ is also forbidden by the $K_1$ bound, a contradiction.

Otherwise we have Case~\ref{IIb} and $c=K_1-1$. But then from metricity of $abc$ either $a+b=c$ (but then $a+b+c$ is even, a contradiction), or $a+b=c+1$ (if $a+b\geq c+2$, then the perimeter of the triangle is too large to be forbidden by the $K_1$ bound). But again in any completion $a'\leq a$ and $b'\leq b$, so either $c'\leq c$ or $c'=c+1$ (from metricity). From Lemma~\ref{lem:sameparity} we know that the parity of $c$ is preserved, hence $c'=c+1$ is absurd, so $a'\leq a$, $b'\leq b$ and $c'\leq c$, and we can apply the same argument as in the preceding paragraph.
\end{proof}

\begin{lemma}[$C_0,C_1$-bound Lemma]
\label{lem:C01bound}
Let $C'>C+1$ and let $\str{G}$, $\overbar{\str{G}}$ be as in Lemma \ref{lem:Cbound}. Then there are no triangles forbidden by either of the $C_0$ and $C_1$ bounds in $\overbar{\str{G}}$.
\end{lemma}

\begin{proof}
Suppose for a contradiction that there is a triangle with vertices $u,v,w$ in $\overbar{\str{G}}$, such that $a+b+c\geq C$ and has parity such that it is forbidden by one of the $C$ bounds, where $a=\bar d(u,v)$, $b=\bar d(v,w)$ and $c=\bar d(u,w)$.

In Case~\ref{IIb}, we have $K_1=K_2$, $C=2K_1+2K_2+1 = 4K_2+1$ and $3K_2=2\delta-1$, hence $C=2\delta + K_2$. For parameters from Case~\ref{III} we have $C\geq 2\delta+K_2$, which means that we always have $C\geq 2\delta+K_2$. This implies that $b,c>K_2\geq M$ and $a\geq K_2$. If $a$ was already present in $\str{G}$, then by Lemmas \ref{lem:misgood}, \ref{lem:bestcompletion} and \ref{lem:sameparity} we have that any completion $a',b',c'$ has $a'=a$, $b'\geq b$ and $c'\geq c$ and the parities are preserved, hence $a',b',c'$ is forbidden by the $C$ bound as well, a contradiction to $\str{G}$ having a completion. If $a$ is not in $\str{G}$, we have $a\neq M$ (by Lemma~\ref{lem:misgood}) and actually $a>M$ as $a\geq K_2\geq M$. Thus we can again use Lemmas \ref{lem:bestcompletion} and \ref{lem:sameparity} to get a contradiction.
\end{proof}

\begin{lemma}[$K_2$-bound Lemma]
\label{lem:K2bound}
Let $\str{G}$, $\overbar{\str{G}}$ be as in Lemma \ref{lem:Cbound}.  Then there are no triangles forbidden by the $K_2$-bound in $\overbar{\str{G}}$.
\end{lemma}

\begin{proof}
We know that all triangles are metric and not forbidden by the $C$ bounds. Suppose for a contradiction that there is a triangle with vertices $u,v,w$ in $\overbar{\str{G}}$ such that $b+c\geq 2K_2+a+1$, where $a=\bar d(u,v)$, $b=\bar d(v,w)$ and $c=\bar d(u,w)$. We know that $b,c > K_2$ and $a\leq K_1$ by Section~\ref{sec:forbtriangles}, where equality can occur only in Case~\ref{III} when $2K_2+K_1=2\delta-1$ and furthermore $M > a$ (because of Definition~\ref{defn:magiccompletion}).

Note that from the conditions for Case~\ref{III}, we know that if $2K_2+K_1=2\delta-1$, then $C\geq 2\delta+K_1+2$, which means that for edges $a,b,c$ Lemma~\ref{lem:sameparity} guarantees that the parity is preserved.

Unless $a=K_1-1$ and Case~\ref{IIb}, Lemmas~\ref{lem:bestcompletion} and~\ref{lem:sameparity} yield that $a'+b'+c'$ has the same parity as $a+b+c$ for any completion $a',b',c'$ and $b'\geq b$, $c'\geq c$ and $a'\leq a$, hence triangle $u,v,w$ is forbidden by the $K_2$ bound in any completion of ${\str{G}}$, which is a contradiction.

The last case remaining is $a=K_1-1=K_2-1$, Case~\ref{IIb}. But then $b+c\geq 2K_2+a+1 = 3K_2=2\delta-1$, so either $b+c=2\delta-1$, or $b+c=2\delta$. But from being forbidden by the $K_2$-bound we know that $a+b+c$ is odd, hence $b+c$ has different parity than $a$. And we know that $a=K_2-1 = \frac{2\delta-1}{3}-1$, which is even, hence $b+c=2\delta-1$. We also know that parities are preserved, so if $a'\geq K_2+1$, then $a'+b'+c'\geq 2\delta + K_2$ and it is thus forbidden by one of the $C$ bounds.
\end{proof}
\begin{proof}[Proof of Theorem~\ref{thm:magiccompletion}]
From the Lemmas~\ref{lem:Cbound}, \ref{lem:metric}, \ref{lem:K1bound}, \ref{lem:C01bound}, \ref{lem:K2bound} we conclude that the algorithm will correctly complete every graph $\str{G}$
which has completion into $\Aclass$. The optimality statement follows by Lemma~\ref{lem:bestcompletion}. Automorphisms are preserved according to Lemma~\ref{lem:aut}.
\end{proof}

%% file: ramseyeppa.tex
\chapter{Ramsey expansion and EPPA}\label{ch:ramseyeppa}

The goal of this chapter is to prove the main results of this thesis:
\begin{theorem}
\label{thm:regularramsey}
For every choice of admissible primitive parameters $\delta$, $K_1$, $K_2$, $C_0$ and $C_1$, the class $\overrightarrow{\mathcal A}^\delta_{K_1,K_2,C_0,C_1}$ of all possible linear orderings of structures from $\mathcal A^\delta_{K_1,K_2,C_0,C_1}$ is Ramsey and
has the expansion property.
\end{theorem}
\begin{theorem}
\label{thm:regulareppa}
For every choice of admissible primitive parameters $\delta$, $K_1$, $K_2$, $C_0$ and $C_1$, the class $\mathcal A^\delta_{K_1,K_2,C_0,C_1}$  has EPPA.
\end{theorem}

To prove them, we employ Theorem~\ref{thm:hn} and Corollary~\ref{cor:herwiglascar} respectively. And in order to do that, we need the following two corollaries of Theorem \ref{thm:magiccompletion}.
\begin{lemma}[Finite Obstacles Lemma]
\label{lem:obstacles}
Let $\delta$, $K_1$, $K_2$, $C_0$ and $C_1$ be primitive admissible parameters.
Then the class $\mathcal A^\delta_{K_1,K_2,C_0,C_1}$ has a finite set of obstacles all of which are cycles of diameter at most $2^\delta \cdot 3$.
\end{lemma}
\begin{example}
Consider $\mathcal A^5_{3,3,16,13}$ discussed in Section~\ref{sec:algorithm}.
The set of obstacles of this class contains all the forbidden triangles listed
earlier, but in addition to that it also contains some cycles with 4 or more vertices. A complete
list of those can be obtained by running the algorithm backwards from the forbidden
triangles.

All such cycles with 4 vertices can be constructed from the triangles by substituting distances by the forks depicted at Figure~\ref{fig:forks}. This means substituting $2$ for $11$ or $55$, and $4$ for $15$ or $51$. With equivalent cycles removed this gives the following list: 
$$
\begin{array}{rrcl}
\hbox{non-metric:}&124&\implies& 1\mathbf{11}4, 1\mathbf{55}4, 12\mathbf{15}, 12\mathbf{51}\\
&125&\implies& 1\mathbf{11}5, 1\mathbf{55}5^{**}\\
&114&\implies& 11\mathbf{15}^*\\
&225&\implies& \mathbf{11}25^{*}, \mathbf{55}25\\
\hbox{$K_1$-bound:}&122&\implies& 1\mathbf{11}2, 1\mathbf{55}2\\
\hbox{$K_2$-bound:}&144&\implies& 1\mathbf{15}4, 1\mathbf{51}4, 14\mathbf{15}\\
&245&\implies& \mathbf{11}45, \mathbf{55}45, 2\mathbf{15}5^*, 25\mathbf{15}\\
\hbox{$C$-bound:}&445&\implies& \mathbf{15}45, \mathbf{51}45, 4\mathbf{15}5\\
\end{array}
$$
Observe that running the algorithm may produce multiple forbidden triangles which leads
to duplicated cycles in the list.
Such duplicates are denoted by $*$. For example, 125 was expanded to $1\mathbf{11}5$. The algorithm will first notice the fork $(1,5)$ and produce $114$. This is also a forbidden triangle, but a different one.  In the case of $1\mathbf{55}5$ (another expansion of 125) the algorithm will again use the fork $(1,5)$ first and produce the triangles $455$ and $145$, which are valid triangles, see Figure~\ref{fig:1555}. Not all expansions here are necessarily forbidden, because not all of them correspond to a valid run of the algorithm.  However with the exception of cases denoted by $**$ all the above 4-cycles are forbidden.
\begin{figure}[t]
\centering
\includegraphics{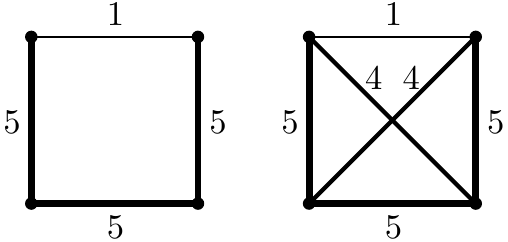}
\caption{Completing the cycle 1555.}
\label{fig:1555}
\end{figure}

Repeating the procedure one obtains the following cycles with five edges that cannot be completed into this class of metric graphs:
 $$11111,\allowbreak 11115,\allowbreak  11155,\allowbreak  11515,\allowbreak  11555,\allowbreak 15155 ,\allowbreak  15555,\allowbreak  55555.$$ 

Because there are no distances 2 or 4 in the cycles with five edges that cannot be completed into this class, it follows that all cycles with at least six edges can be completed.

An example of a failed run of algorithm trying to complete one of the forbidden cycles is depicted in Figure~\ref{fig:11555}.
\begin{figure}
\centering
\includegraphics{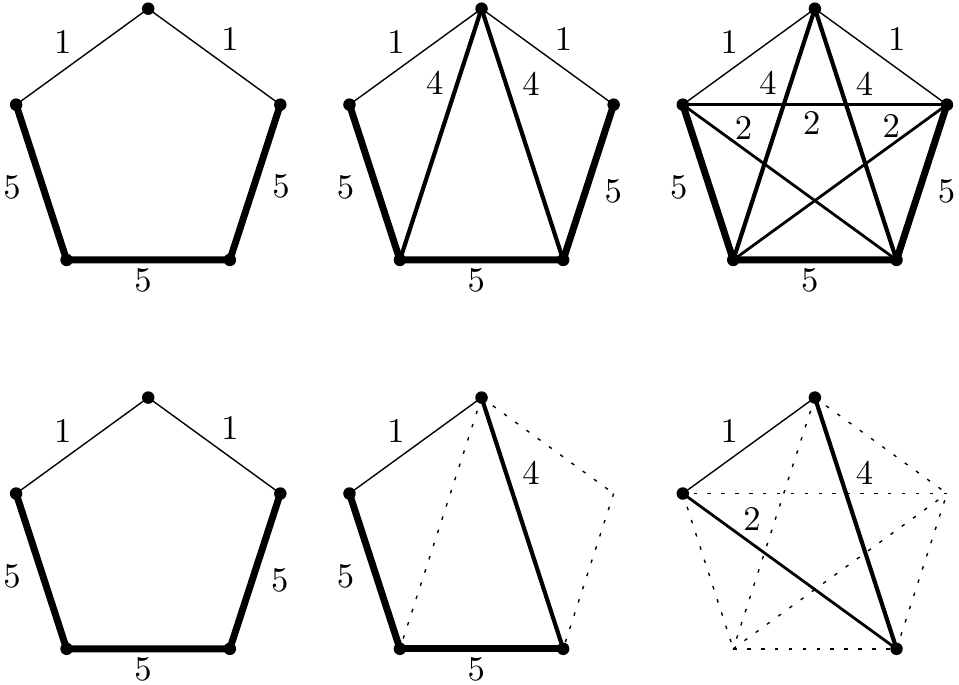}
\caption{Failed run attempting to complete the cycle 11555. In the bottom row is the backward run from the non-metric triangle 124 to the original obstacle used in the proof of Lemma~\ref{lem:obstacles}.}
\label{fig:11555}
\end{figure}
\end{example}
 
\begin{proof}[Proof of Lemma \ref{lem:obstacles}]
Let $\str{G}=(G,d)\in \mathcal G^\delta$ be an edge-labelled graph no completion in
$\mathcal A^\delta_{K_1,K_2,C_0,C_1}$.  We seek a subgraph of $\str{G}$ of bounded size
which has also no completion into $\mathcal A^\delta_{K_1,K_2,C_0,C_1}$.

Consider the sequence of graphs $\str{G}_0, \str{G}_1,\ldots,\str{G}_{2M+1}$ as given by Definition~\ref{defn:magiccompletion}
when completing $\str{G}$ with magic parameter $M$.  Set $\str{G}_{2M+2}$ to be the actual completion.

Because $\str{G}_{2M+2}\notin \mathcal A^\delta_{K_1,K_2,C_0,C_1}$ we know it contains a forbidden triangle $\str{O}$.
By backward induction on $k=2M+1,2M,\ldots, 0$ we obtain cycles $\str{O}_k$
 such that $\str{O}_k$ has no completion in $\mathcal A^\delta_{K_1,K_2,C_0,C_1}$
and there exists a homomorphism $f_k\colon \str{O}_k\to \str{G}_k$. We assume that all $\str{G}_k$ share the vertex set, thus $f_k$ makes sense as a function from $\str{O}_k$ to $\str{G}_{k'}$ for any $k'$ (although it is not necessarily a homomorphism).

Put $\str{O}_{2M+1}=\str{O}$. By Lemma~\ref{lem:misgood} we know that this triangle is also in $\str{G}_{2M+1}$.
At step $k$ we will create $\str{O}_{k}$ by changing $\str{O}_{k+1}$: Consider every edge $u,v$ of $\str{O}_{k+1}$ which is not an edge of $\str{G}_k$ (by this we mean that $f_{k+1}(u)$ and $f_{k+1}(v)$ form an edge in $\str{G}_{k+1}$, but the corresponding vertices do not form an edge of $\str{G}_k$), look at
its witness $w$ (i.e. vertex $w$ in $\str{G}_k$ such that the edges $f_k(u),w$ and $f_k(v),w$ implied the addition of the edge $f_k(u),f_k(v)$ to $\str{G}_{k+1}$) and replace the edge $uv$ in $\str{O}_k$ with the fork $uw'v$, where $w'$ is a new vertex with $d_{\str{O}_{k}}(u,w')=d_{\str{G}_{k+1}}(u,w)$ and $d_{\str{O}_{k}}(v,w')=d_{\str{G}_{k+1}}(v,w)$.
One can verify that the completion algorithm will fail to complete $\str{O}_k$ the same way
as it failed to complete $\str{O}_{k+1}$ and moreover there is a homomorphism $\str{O}_{k+1}\to \str{G}_{k+1}$.

At the end of this procedure we obtain $\str{O}_0$, a subgraph of $\str{G}$, that
has no completion into $\mathcal A^\delta_{K_1,K_2,C_0,C_1}$.
The bound on the size of the cycle follows from the fact that only $\delta$ steps
of the algorithm are actually changing the graph and each time every edge may
introduce at most one additional vertex.

Let $\mathcal O$ consist of all edge-labelled cycles with at most $2^\delta \cdot 3$
vertices that are not completable in $\mathcal A^\delta_{K_1,K_2,C_0,C_1}$. Clearly $\mathcal O$ is finite. To check that $\mathcal O$ is a set of obstacles it remains
to verify that there is no $\str{O}\in \mathcal O$ with a homomorphism
to some $\str{M}\in \mathcal A^\delta_{K_1,K_2,C_0,C_1}$. Denote by $\mathcal{O}'$ the set of all homomorphic images
of structures in $\mathcal{O}$ that are not completable in $\mathcal A^\delta_{K_1,K_2,C_0,C_1}$.
Assume, to the contrary, the existence of such an
$\str{O}=(O,d)\in \mathcal{O}'$ and $\str{M}=(M,d')\in\Aclass$ and a homomorphism $f\colon \str{O}\to\str{M}$ and among all those choose one minimising 
the difference of the numbers of their vertices, $\lvert O\rvert-\lvert M\rvert$. From minimality it follows that $\lvert O\rvert-\lvert M\rvert=1$, hence $f$ is surjective and there is exactly one pair of vertices $x,y\in O$ such that $f(x)=f(y)$.

Let $\str{M}'=(M',d'')$ be the substructure of $\str{M}$, where $M'=M\setminus\{f(x)\}$. And let $\alpha\colon \str M' \rightarrow \str M$ be the inclusion embedding. As $\mathcal A^\delta_{K_1,K_2,C_0,C_1}$ has the strong amalgamation property and both $\str M$ and $\str M'$ are in $\Aclass$, there is $\str K$ and embedding $\beta_1\colon \str M\rightarrow \str K$ and $\beta_2\colon  \str M\rightarrow \str K$, such that $\beta_1\circ\alpha = \beta_2\circ \alpha$ and furthermore $\beta_1$ and $\beta_2$ overlap only by $\alpha(M')$. We can further assume that $K = M'\cup\{x,y\}$ and $\beta_1\circ\alpha$ is an inclusion.

In human terms this means that if one un-glues the vertices $x$ and $y$ of $\str M$, from the strong amalgamation property one gets a structure $\str K\in \Aclass$ which is a completion of $\str M$ with $x$ and $y$ unglued. But then clearly $\str K$ is a completion of $\str O$ in $\Aclass$, which is a contradiction.
\end{proof}

\begin{proof}[Proof of Theorem~\ref{thm:regularramsey}]
Ramsey property follows by a combination of Theorem~\ref{thm:hn} and Lemma~\ref{lem:obstacles} in the same way as it does for metric spaces in Corollary~\ref{cor:metricareramsey}.

To show the expansion property we use the now standard argument that edge-Ramsey implies ordering property~\cite{Nevsetvril1996,Jasinski2013}:
Given a metric space $\str{A}\in \mathcal A^{\delta}_{K_1,K_2,C_0,C_1}$ construct an ordered metric space $\overrightarrow{\str{B}}_0\in \overrightarrow{\mathcal A}^\delta_{K_1,K_2,C_0,C_1}$ as a disjoint union of all
possible linear orderings of $\str{A}$. Now consider every pair of
vertices $u < v$, $d(u,v)\neq M$ and add third vertex $w$ in distance $M$ from both $u$ and $v$ with the order extended in a way so
$u < w < v$ holds. Because $M$ is magic, by Observation~\ref{obs:magicismagic}, all new triangles are allowed and thus it is possible to complete this structure to an ordered metric space $\overrightarrow{\str{B}}_1\in \overrightarrow{\mathcal A}^\delta_{K_1,K_2,C_0,C_1}$.
Now denote by $\overrightarrow{\str{E}}$ the ordered metric space consisting of two vertices in distance $M$ and construct
$$\overrightarrow{\str{B}}\longrightarrow(\overrightarrow{\str{B}_1})^{\overrightarrow{\str{E}}}_2.$$

We claim that $\str{B}$ (the unordered reduct of $\overrightarrow{\str{B}}$) has the property that every ordering of $\str{B}$
contains every ordering of $\str{A}$.  Denote by $\leq$ the order of $\overrightarrow{\str{B}}$ and choose an arbitrary linear order order $\leq'$ of vertices of $\str{B}$. $\leq'$ implies two-coloring of copies of $\overrightarrow{\str{E}}$
in $\overrightarrow{\str{B}}$: color a copy red if both orders agree and blue otherwise.  Because $\overrightarrow{\str{B}}$ is Ramsey, we obtain a monochromatic
copy of $\overrightarrow{\str{B}}_1$ which contains a copy of $\overrightarrow{\str{B}}_0$ with the property that $\leq'$ restriced to this copy either agrees either with $\leq$ or with $\geq$.
In both cases we obtain a copy of every ordering of $\str{A}$ within this copy of $\overrightarrow{\str{B}}_0$.
\end{proof}

\begin{proof}[Proof of Theorem~\ref{thm:regulareppa}]
Follows from Corrolary~\ref{cor:herwiglascar}, Lemma~\ref{lem:obstacles} and Lemma~\ref{lem:aut}.
\end{proof}

%% file: epilog.tex
\chapter{Conclusion}\label{ch:conclusion}
We found Ramsey expansions with the expansion property of all the primitive 3-constrained classes from Cherlin's catalogue of metrically homogeneous graphs. This was done by devising an explicit way to fill in the missing distances in $\delta$-edge-labelled graphs to obtain Cherlin's metric spaces. This method also implies EPPA. (It is one of the first times when EPPA and Ramseyness were proved at the same time.) This result is a contribution to the Ne\v set\v ril classification programme of Ramsey classes and also gives more insight into the catalogue.

To our surprise, the techniques (in particular the magic completion algorithm) turned out to be very flexible and generalise to many more classes then just the primitive 3-constrained ones:

\section{The rest of the catalogue}
In addition to the primitive classes, Cherlin's catalogue also contains:
\begin{description}
\item[Bipartite classes:] Bipartite classes are an extremal variant of the primitive 3-constrained $\Aclass$ where all odd triangles are forbidden, which one can model for example by setting $K_1 = \infty$.
\item[Antipodal classes:] These are another extremal variant of $\Aclass$ where $C=2\delta+1$ and hence all triangles with two edges of length $\delta$ are forbidden. This means that edges of length $\delta$ form a matching. Antipodal classes can also be bipartite.
\item[Infinite diameter:] Bipartite and primive classes can also have $\delta=\infty$, but they are Ramsey and have EPPA by a simple reduction to the finite-diameter variants.
\item[Tree-like structures:] The tree-like structures are metric spaces of infinite diameter where the edges of length one form complete graphs such that each complete graph has size at most $m$, each vertex is in at most $n$ such cliques and each vertex is a cutvertex.
\end{description}

Besides this, one can also observe that the magic completion never introduces edges of length $\delta$ and $1$ (unless $M=\delta$ and $C=2\delta+2$ respectively) and thus, in the primitive cases, one can also further forbid any set of finite metric spaces containing only distances 1 and $\delta$ --- Cherlin calls these \emph{Henson's constraints}.

Throughout the thesis, we mentioned several times that the primitive 3-constrained classes are ``core'', ``key'' etc. It probably deserves a proper justification, which we can now give. For a precise statement of the results (which we just sketch in the following paragraphs) see~\cite{Aranda2017}.

Initially when we started on this project, the goal was to understand the primitive 3-constrained cases, because the other cases intuitively seemed to be extremal variants. Later it turned out, rather surprisingly, that the magic completion algorithm quite naturally extends to all the other non-tree-like cases.

In the bipartite cases one needs to make sure not to introduce any odd cycles. And if one is careful, the only step in which odd cycles could be introduced is the last one, when adding distances $M$. Our solution is to have a magic pair of distances, $M$ and $M+1$ and add them in the last step according to whether the edge goes across the bipartitions or inside them. The Ramsey expansion is slightly more complex (although standard), as one needs to distinguish the bipartitions (by, say, two unary relations) and only take orders \emph{convex} with respect to the bipartitions. The proofs for the bipartite magic completion algorithm are direct analogues of the proofs for the magic completion algorithm and they are often much simpler and shorter.

The antipodal spaces contain a matching; for each vertex there is at most one vertex in distance $\delta$ from it (in the \Fraisse{} limit exactly one). But one can without doing any harm assume that each vertex has exactly one in distance $\delta$ from it. Furthermore, in these cases it holds that if one has two pairs of vertices in distance $\delta$, then a single distance between the pairs determines all three other distances. So again, one can assume that if one distance is defined between two edges of length $\delta$, then all four are defined. Thus one can pick one representant from each $\delta$-pair and this induces a $(\delta-1)$-edge-labelled graph where the magic completion algorithm can be used. The Ramsey expansion is slightly more complex (besides other things, one needs to fix the choice of the representants by for example expanding the language by a unary mark which selects exactly one vertex from each edge of length $\delta$), but again standard.

The bipartite antipodal spaces need a mixture of both techniques.

While for the Ramsey property everything is settled, for EPPA this is not the case, because for certain antipodal classes it is not known whether they have EPPA\footnote{They do have EPPA when expanded by unary marks for the representants.}: To use the magic completion algorithm, we needed to pick a representant of each $\delta$-pair. And sometimes different choices lead to different completions, which are not automorphism preserving. The smallest $\delta$ for which this happens is $\delta=3$ and the class of all finite antipodal Cherlin's spaces of diameter $3$ is bi-interpretable with \emph{two-graphs} (triple systems such that on every four vertices there is an even number of triples, they have been studied since the 1960's~\cite{Seidel1973}). The question whether two-graphs have EPPA has been asked by Macpherson and also appears in~\cite{Siniora2}.

\section{Explicit description of forbidden cycles}
Lemma~\ref{lem:obstacles} which we proved in Chapter~\ref{ch:ramseyeppa} states that there is a finite family $\mathcal F^\delta_{K_1, K_2, C_0, C_1}$ of $\delta$-edge-labelled cycles such that a $\delta$-edge-labelled graph $\str G$ has a completion in $\Aclass$ if and only if $\str G\in \Forb\left(\mathcal F^\delta_{K_1, K_2, C_0, C_1}\right)$. However, it did not give an explicit description of $\mathcal F^\delta_{K_1, K_2, C_0, C_1}$. In~\cite{Coulson2018}, Coulson, Hubi\v cka, Kompatscher and the author prove that $\mathcal F^\delta_{K_1, K_2, C_0, C_1}$ is the union of the following cycles:
\begin{description}
\item[Non-metric cycles:] Cycles with edges $a, x_1, x_2, \ldots, x_k$ such that $$a > \sum_i x_i.$$
\item[$C_0$-cycles:] Cycles of even perimeter with distances $d_0, d_1, \ldots, d_{2n}, x_1, \ldots, x_k$ for some $n\geq 0$ and $k\geq 0$ such that
$$\sum d_i > n(C_0 - 1) + \sum x_i.$$
\item[$C_1$-cycles:] Cycles of odd perimeter with distances $d_0, d_1, \ldots, d_{2n}, x_1, \ldots, x_k$ for some $n\geq 0$ and $k\geq 0$ such that
$$\sum d_i > n(C_1 - 1) + \sum x_i.$$
\item[$K_1$-cycles:] \emph{Metric} cycles of odd perimeter with distances $x_1, \ldots, x_k$ such that $$2K_1 > \sum_i x_i.$$
\item[$K_2$-cycles:] Cycles of odd perimeter with distances $d_1, \ldots, d_{2n+2}, x_1, \ldots, x_k$ such that $$\sum d_i > n(C-1) + 2K_2 + \sum x_i.$$
\end{description}
Note that the non-metric cycles are precisely the union of $C_0$- and $C_1$-cycles for $n=0$.

\section{Future work}
There are still some open questions remaining and worth investigating. We have already mentioned the question whether two-graphs have EPPA. It seems like a difficult question: The personal opinion of the author (shared by others) is that two-graphs probably do not have EPPA. And there are not many techniques to show that a class does not have EPPA. As EPPA implies that the automorphism group of the \Fraisse{} limit is amenable, one possibility is to prove non-amenability. But in Section~9.2 of~\cite{Aranda2017} we show that our Ramsey expansions are ``regular enough'' to actually give amenability for the non-expanded structures.

In~\cite{Aranda2017} we also show that the families of forbidden substructures for the antipodal and bipartite cases are ``nicely behaved'' (for the bipartite cases all odd cycles are obviously forbidden and there are infinitely many of them, but this is the only complication). It would make sense to extend the explicit list of forbidden cycles also to these cases.

Tent and Ziegler~\cite{Tent2013, Tent2013b} proved that the automorphism group of the bounded Urysohn space (distances from the closed interval $[0,1]$) is simple and that the automorphism group of the Urysohn space has only one normal subgroup. Their proof heavily relies on the properties of the shortest path completion and there is hope that one might be able to amend it to work also for the primitive 3-constrained cases. First steps towards this were already taken by Li~\cite{Li2018}, she proves simplicity for classes of diameter at most 4.

%% file: bibliography.tex

 \bibliographystyle{alpha}

\renewcommand{\bibname}{Bibliography}


\bibliography{./bibliography.bib}